\documentclass[12pt, reqno]{amsart}
\usepackage{amsmath}
\usepackage{amsfonts}
\usepackage{amssymb}
\usepackage[english]{babel}
\usepackage{color}

\usepackage{hyperref}
\usepackage{calrsfs}

\usepackage{multicol}
\usepackage{graphicx}

\usepackage[normalem]{ulem}

\usepackage{tikz}
\usetikzlibrary{arrows}

\theoremstyle{plain}
\newtheorem {theorem}{Theorem}[section]
\newtheorem {lemma}[theorem]{Lemma}
\newtheorem {corollary} [theorem]{Corollary}

\newtheorem {proposition} [theorem]{Proposition}
\theoremstyle{definition}
\newtheorem{definition}[theorem]{Definition}
\newtheorem{remark}[theorem]{Remark}

\newtheorem{ass}[theorem]{Conjecture}
\theoremstyle{remark}

\numberwithin{equation}{section}

\newcommand{\N}{\mathbb{N}}
\newcommand{\R}{\mathbb{R}}
\newcommand{\C}{\mathrm{C}}

\newcommand{\cH}{\mathcal{H}}
\newcommand{\cD}{\mathcal{D}}
\newcommand{\cC}{\mathcal{C}}
\newcommand{\n}{\phi^*}
\newcommand{\X} {\mathcal{X}\! f}
\newcommand{\cN}{\mathcal N}
\newcommand{\nB}{\psi}

\renewcommand{\epsilon}{\varepsilon}
\renewcommand{\H}{\mathbb{H}}
\renewcommand{\P} {\mathcal P\!\!}


\begin{document}
\author[V.~Franceschi]{Valentina Franceschi$^1$}

\author[R.~Monti]{Roberto Monti$^1$}

\author[A.~Righini]{Alberto Righini$^1$}

\author[M.~Sigalotti]{Mario Sigalotti$^2$}

\address{$^{1}$Dipartimento di Matematica Tullio Levi Civita, Universit\`a di Padova, Italy}
\email{valentina.franceschi@unipd.it}
\email{monti@math.unipd.it}
\email{righini@math.unipd.it}

\address{$^{2}$Laboratoire Jacques-Louis
Lions, Sorbonne Universit\'e, Inria, CNRS, Université Paris Cit\'e, Paris, France}
\email{mario.sigalotti@inria.fr}

\title{The isoperimetric problem for regular and crystalline norms in $\mathbb H^1$}

\begin{abstract} 
We study the isoperimetric problem for anisotropic perimeter measures on  $\R^3$, endowed with the  Heisenberg group structure. 
The perimeter is associated with a left-invariant norm $\phi$ on the horizontal distribution.
In the case where $\phi$ is the standard norm in the plane, such isoperimetric problem is the subject of Pansu's conjecture, which is still unsolved. 
Assuming some regularity on $\phi$ and on its dual norm $\phi^*$, we characterize $\mathrm{C}^2$-smooth isoperimetric sets as the sub-Finsler analogue of Pansu's bubbles. The argument is based on a fine study of the characteristic set of $\phi$-isoperimetric sets and on establishing a  foliation property by sub-Finsler geodesics. 
When $\phi$ is a crystalline norm, we show the existence of a partial foliation for constant $\phi$-curvature surfaces by sub-Finsler geodesics. By an approximation procedure, we finally prove a conditional minimality property for the candidate solutions in the general case (including the case where $\phi$ is crystalline).
\end{abstract}

\maketitle


\section{Introduction}
Let $\phi:\R^n\to[0,\infty)$ be a norm in $\R^n$, $n\geq 2$. The associated \emph{Finsler} or \emph{anisotropic perimeter} of a Lebesgue measurable set $E\subset\R^n$ is defined as
\[
P_{\phi}(E)=\sup \left \{  \int_{E}
\mathrm{div}(V) \;
dp : V
\in \C^\infty_c(\R^n;\R^n)  \textrm{ with }
\max_{p\in \R^n} \phi(V (p))
  \leq 1 \right \}.
\]
If $E$ is  regular, 
$P_\phi(E)$ can be represented as a surface integral as follows
\[
P_\phi(E)= \int_{\partial E}\phi^*(\nu_E)\;d\mathcal H^{n-1},
\]
where $\nu_E$ is the inner unit normal to $\partial E$ and $\phi^*:\R^n\to[0,\infty)$ is the dual norm defined by
\[
\phi^* (w )  = \max _{\phi(v) =1} \langle w, v\rangle,\quad w\in\R^n.
\]
Here, $\langle\cdot,\cdot\rangle$ denotes the standard Euclidean scalar product in $\R^n$ and
$|\cdot|$
the Euclidean norm.  
In the theory of crystals, $\phi^*$ is the surface tension of the interface between an anisotropic material $E$ and a fluid, and $P_\phi(E)$ is the total free energy.

In the case where $\phi=|\cdot|$, $P_\phi$ is the standard De Giorgi perimeter and
 isoperimetric sets
 (i.e., sets of fixed volume that minimize perimeter) 
 are Euclidean balls.
For a general norm $\phi$, isoperimetric sets are translations and dilations of the \emph{Wulff shape}, first considered by G.~Wulff in \cite{Wulff}. In our   notation it corresponds to 
the unit ball of the $\phi$-norm.  
The first complete proof of the isoperimetric property of Wulff shapes in the class of Lebesgue measurable sets with given volume is contained in \cite{Fonseca91,FM91}, and based on the Brunn-Minkowski inequality. We refer to \cite{FMP} for a quantitative version.

 In this paper, we study the isoperimetric problem for sub-Finsler perimeter measures in the Heisenberg group $\mathbb H^1$.
 The latter  is $\R^3$ endowed with the non-commutative group law
\begin{equation}\label{eq:*Heis}
(\xi,z)*(\xi',z')=\left(\xi+\xi',z+z'+\omega(\xi,\xi')\right),\quad \xi,\xi'\in\R^2,\quad z,z'\in\R,
\end{equation}
where
$\omega:\R^2\times\R^2\to\R$ is the symplectic form
\begin{equation}\label{eq:omega}
\omega(\xi,\xi') = \frac{1}{2}(xy'-x'y),\quad \xi=(x,y),\ \xi'=(x',y')\in\R^2.
\end{equation} 
The vector fields 
\[
X=\frac{\partial}{\partial  x}-\frac{y}{2} \frac{\partial}{\partial  z}
\qquad
\text{and}
\qquad Y=\frac{\partial}{\partial  y}+\frac{x}{2}\frac{\partial}{\partial  z}
\]
are left-invariant for the group action and span a two-dimensional distribution $\cD(\H^1)$ in $T\H^1$, called the \emph{horizontal distribution}.   
We denote by $\cD(p)$ the fiber of $\cD$ at $p\in\H^1$.  

Given a norm $\phi:\R^2\to[0,\infty)$, the associated anisotropic perimeter measure in $\H^1$
is introduced in Definition \ref{d:per} and
 takes into account only horizontal directions. 
 For a regular 
set $E\subset\R^3$ it  can be 
represented as
\[
\P_\phi(E)= \int_{\partial E}\phi^*(N_E)\;d\mathcal H^{2},
\]
where $N_E$ is 
obtained by projecting the inner unit normal $\nu_E$ onto the 
horizontal distribution. 
A set $E\subset\H^1$ is said to be \emph{$\phi$-isoperimetric} if there exists  $m>0$ such that $E$ minimizes
\begin{equation}\label{eq:ip}
\inf\left\{\P_\phi(E) : E\subset \H^1\text{ measurable, }\mathcal L^3(E)=m\right\}.
\end{equation}

If $\phi=|\cdot|$ is the Euclidean norm in $\R^2$, then $\P_\phi$ corresponds to the standard horizontal perimeter in $\H^1$, introduced and studied in \cite{CDG94,GN96,FSSC96}. In this case, the problem of characterizing $\phi$-isoperimetric sets in the class of Lebesgue measurable sets in $\H^1$ is open. According to Pansu's conjecture \cite{P82}, $|\cdot|$-isoperimetric sets are obtained through
\emph{left-translations} 
and \emph{anisotropic dilations} $\delta_\lambda:\H^1\to\H^1$, $\lambda>0$,
\begin{equation}\label{eq:dilat}
\delta_\lambda(\xi,z)=(\lambda \xi,\lambda^2 z),
\end{equation}
of the so-called \emph{Pansu's bubble}, that we now present.

An absolutely continuous curve $\gamma:I\to\H^1$ is said to be \emph{horizontal} if $\dot\gamma(t)\in\cD(\gamma(t))$ for a.e.\ $t\in I$ and we call \emph{horizontal lift} of an absolutely continuous curve $\xi:I\to\R^2$ any horizontal curve $\gamma=(\xi,z)$ with
\begin{equation*}\label{eq:hlift}
\dot z=\omega(\xi,\dot\xi).
\end{equation*}
Pansu's bubble is the bounded set whose boundary is foliated by horizontal lifts of planar circles of a given radius, passing through the origin.
Such 
horizontal curves are length minimizing between their endpoints for the sub-Riemannian distance in $\H^1$, so that Pansu's conjecture in $\H^1$ explicits a relation between isoperimetric sets and the geometry of the ambient space.
The conjecture is supported by several results contained in \cite{RR08,M08,MR09,R12,FM16,FMM19}, but it is still unsolved  in its full generality. A quantitative version of the Heisenberg isoperimetric inequality has been proposed in \cite{FLM15}.

Very little is known on the isoperimetric problem when $\phi:\R^2\to[0,\infty)$ is a general norm in $\R^2$, apart from \cite{Sanchez-thesis}  for the statement of the problem and \cite{PozueloRitore} for a calibration result in suitable half-cylinders.
Existence of $\phi$-isoperimetric sets
follows by the  arguments of \cite{LR03},  see Section~\ref{s:ex}.
The construction  of the Pansu's bubble can be generalized  to the sub-Finsler context in the following way.
We call \emph{$\phi$-circle} of radius $r>0$ and center $\xi_0\in\R^2$ the set
\begin{equation}\label{eq:qCircle}
       C_\phi(\xi_0,r)  = \{\xi  \in \R^2: \phi(\xi-\xi_0) =r  \},
\end{equation}
and we call  \emph{$\phi$-bubble} the bounded set $E_\phi$ whose
boundary is foliated by horizontal lifts of $\phi$-circles in the plane of a given radius, passing through the origin.

In Figure~\ref{f:bubbles} we represent two $\phi$-bubbles, corresponding to $\phi=\ell^p$, with $\ell^p(x,y)=(|x|^p+|y|^p)^{\frac{1}{p}}$, 
 in the cases $p=3$ and $p=100$. The latter can be seen as an approximation of the  
crystalline case.

\begin{figure}
\includegraphics[width=.4\textwidth]{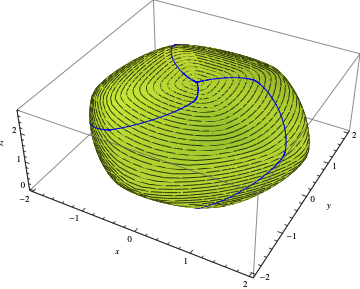}\hspace{1cm}
\includegraphics[width=.44\textwidth]{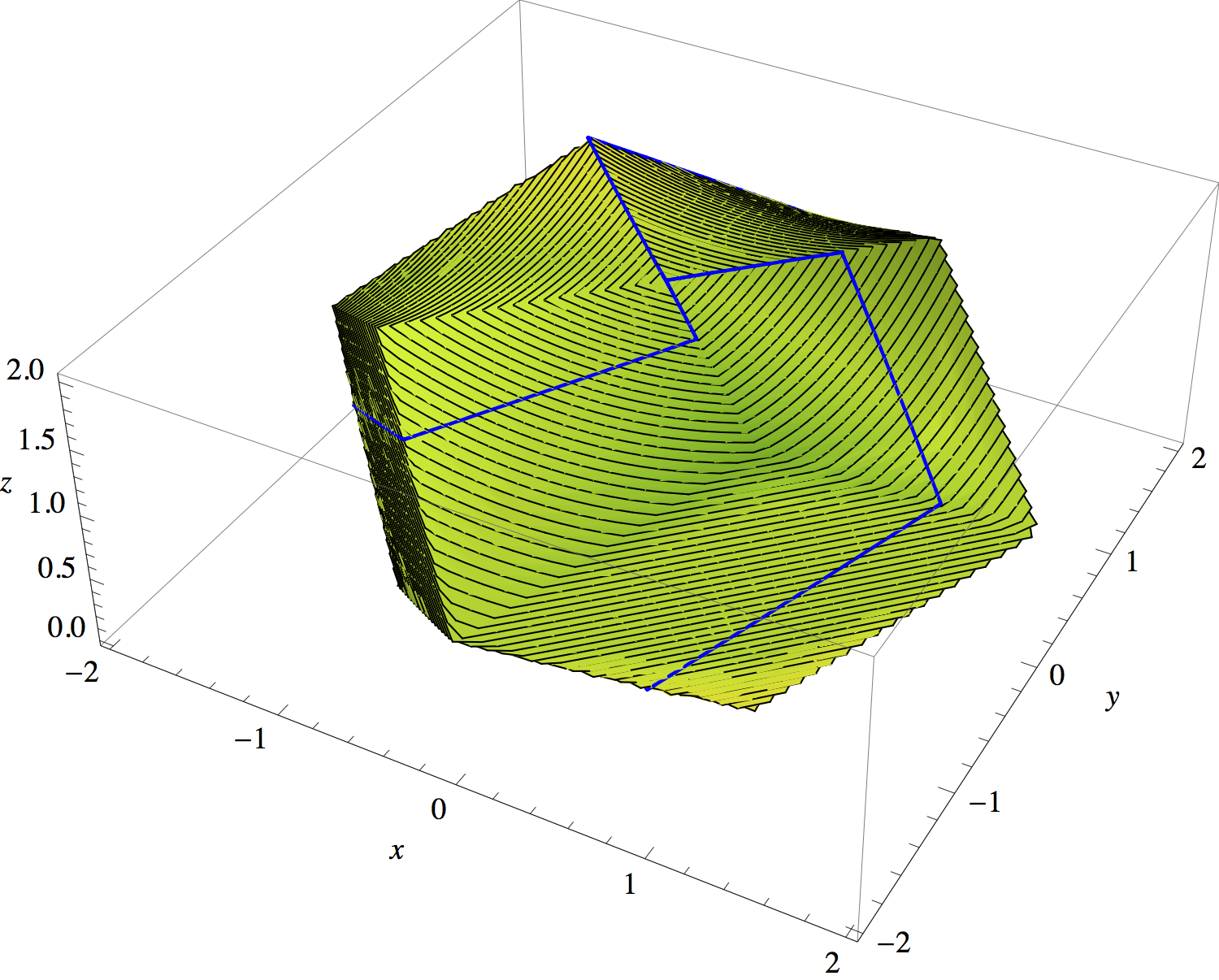}
\caption{The $\ell^{p}$-bubbles with $p=3$ (left) and $p=100$ (right).
In blue we outlined 
three horizontal lifts of $\ell^{p}$-circles foliating the $\ell^{p}$-bubble.
}
\label{f:bubbles}
\end{figure}

    Our main result is the characterization of $\C^2$-smooth $\phi$-isoperimetric sets when $\phi$ and $\phi^*$ are $\C^2$-smooth.
    This result suggests that the $\phi$-bubble is the  solution to the isoperimetric problem for $\P_\phi$.
    Here and in the following, if $\phi\in \C^k(\R^2\setminus\{0\})$ we say that $\phi$ is of class $\C^k$, for $k\in \N$.
    The $\C^2$ regularity of both $\phi$ and $\phi^*$ can be reformulated in terms of $\C^2$ regularity of $\phi$ and an additional positivity constraint on the curvature of $\phi$-circles, see Proposition~\ref{prop:c2+}.

\begin{theorem}\label{thmi:class}
Let $\phi$ be a norm of class $\C^2$ such that $\phi^*$ is  of class $\C^2$. If $E\subset \mathbb H^1$ is a $\phi$-isoperimetric set of class $\C^2$ then  we have $E=E_\phi$, up to
left-translations and anisotropic dilations.
\end{theorem}

The proof of Theorem~\ref{thmi:class} is presented in Section~\ref{s:proof} and  is based on a fine
study of the \emph{characteristic set} of isoperimetric sets. 
The characteristic set
of a set  $E\subset\H^1$ of class $\C^1$ (equivalently, of its boundary $\partial E$) is
\begin{equation}\label{eq:char}
\cC(E){=\cC(\partial E)}=\{p \in \partial E : T_p \partial E = \cD(p)\}.
\end{equation}
In Section~\ref{s:char}  we characterize the structure of $\cC(E)$ for a $\C^2$-smooth $\phi$-isoperimetric set $E\subset\H^1$, proving that $\cC(E)$ is made of isolated points. For the more general case of $\phi$-critical surfaces we obtain the following result, that we prove  {by} adapting to the sub-Finsler case the theory of Jacobi fields of \cite{RR08}. Any $\phi$-critical surface has  constant $\phi$-curvature and the definition is presented in Section~\ref{s:char}. 
  
\begin{theorem}\label{thmi:char}
Let $\phi$ and $\phi^*$ be of class $\C^2$ 
and let $\Sigma\subset \H^1$ be a complete and   oriented  surface of class $\C^2$. If $\Sigma$ is $\phi$-critical with non-vanishing $\phi$-curvature then $\cC(\Sigma)$ consists of isolated points and $\C^2$ curves that are either horizontal lines or horizontal lifts of simple closed curves.
\end{theorem}
The simple closed curve{s} of Theorem~\ref{thmi:char} are 
described by a  {suitable ordinary }
differential equation. We expect 
that  these curves are $\phi^\dagger$-circles, where $\phi^\dagger$ is the  norm
defined as
\begin{equation*}
\phi ^\dag (\xi)  = \phi^*(\xi^\perp),\quad \xi\in\R^2.
\end{equation*}
Here and hereafter, $\perp:\R^2\to\R^2$ denotes the perp-operator $\perp\!\!(\xi) =\xi^\perp$, with
\[
\xi^\perp = (x,y)^\perp = (-y,x),\quad \xi=(x,y)\in\R^2.
\]

Theorem~\ref{thmi:class} then follows by combining the results of  {Sections~\ref{ss:fv-reg}, \ref{s:integ}, and \ref{s:char}}.
In particular, starting from  a first variation analysis, we
establish a foliation property outside the characteristic set for $\C^2$-smooth $\phi$-isoperimetric sets (and more generally for constant $\phi$-curvature surfaces). {Theorem \ref{thmi:char} is a key step for concluding the proof.}

We also
identify an
explicit 
 relation between $\phi$-isoperimetric sets and geodesics in the ambient space. In Section~\ref{s:PMP}, we show that, outside the characteristic set, $\phi$-isoperimetric sets are foliated
by \emph{sub-Finsler geodesics} in $\H^1$ relative to the norm
$\phi^\dag$.
We refer to Corollary~\ref{c:isop-lm} for a statement of the result.
Notice that when $\phi=|\cdot|$ is the Euclidean norm,
$\phi^\dagger$ reduces to $|\cdot|$, and we recover the 
foliation by sub-Riemannian geodesics of $\C^2$-smooth
$|\cdot|$-isoperimetric sets.

In the case where $\phi$ or $\phi^*$ are 
{not differentiable},
Theorems~\ref{thmi:class} 
cannot be   applied
in a direct way. In Corollary~\ref{rem:pieceC2} we show that the foliation property by horizontal lifts of $\phi$-circles outside the characteristic set can be recovered when $\phi^*$ is only \emph{piecewise $\C^2$}, thus allowing to cover the case
$\phi=\ell^p$ 
 for $p>2$.
For general non-differentiable norms, our next result
is conditioned to the validity of the following conjecture.

\begin{ass} \label{iBenbow}
For any norm $\phi$ of class $\C^\infty$ such that $\phi^*$ is of class $\C^\infty$, $\phi$-isoperimetric sets are of class $\C^2$.
\end{ass}

It would actually be natural to extend the conjecture to any $\phi$ of class $\C^2$ with dual norm of class $\C^2$, but we choose the $\C^\infty$ hypothesis in order  to have weaker assumptions in the next result, whose proof is presented in Section~\ref{s:cr}.

\begin{theorem} \label{thmi:appr}
{
Assume that Conjecture \ref{iBenbow} holds true.
Then   for any norm $\phi$ in $\R^2$
the $\phi$-bubble $E_\phi\subset \mathbb H^1$  is $\phi$-isoperimetric.
}
\end{theorem}

{
Of a particular interest is the case of a \emph{crystalline norm}.
A norm
 $\phi:\R^2\to[0,\infty)$ is called crystalline if  the $\phi$-circle
  $C_\phi=C_\phi(0,1)$ is a   convex polygon centrally symmetric with respect to the origin.
 Let $v_1,\ldots,v_{2N}\in C_\phi$ be the ordered vertices of
this polygon, and denote by $e_i = v_i-v_{i-1}$, $i=1,\ldots,2N$,
the edges of $C_\phi$, where $v_0 =v_{2N}$.
We consider the left-invariant vector fields
\begin{equation}\label{eq:X_i}
X_i:=e_{i,1}X+e_{i,2}Y,\quad i=1,\ldots, 2N,
\end{equation}
 where $e_i =(e_{i,1}, e_{i,2})$, and we notice that $X_{i+N}=-X_i$ for $i=1,\dots,N$.
By a first variation argument, we deduce a foliation
   property for $\phi$-isoperimetric sets
   by integral curves of the $X_i$, see Section~\ref{ss:fv-cr}.

\begin{theorem}\label{thmi:folz}
Let $E\subset\H^1$ be  $\phi$-isoperimetric for
a crystalline norm $\phi$. Let  $A\subset\H^1$ be
an open  set
such that $\partial E\cap A$ is a connected $z$-graph of  class  $\C^2$.
Then there exists $i=1,\dots,N$ such that $\partial E\cap A$ is foliated by integral curves of
$X_i$.
\end{theorem}

Geodesics of 
sub-Finsler structures on the Heisenberg group and other Carnot groups 
have been studied in several papers (see, in particular, 
\cite{ALDS,BBLS,Ber94,L19,S20}).  
Unfortunately,
Theorem \ref{thmi:folz}
does not provide enough  information in order to establish the global
foliation property by $\phi^\dagger$-geodesics in the   crystalline case.

\subsection{Structure of the paper} In Section~\ref{s:per} we introduce the notion of sub-Finsler perimeter and we deduce a representation formula for Lipschitz sets (see Proposition~\ref{p:repr}), holding for any norm $\phi$ in $\R^2$. In Section~\ref{s:ex} we prove existence of $\phi$-isoperimetric sets for a general norm $\phi$, following the arguments in \cite{LR03}. In Section~\ref{s:fv} we derive first-variation necessary conditions for $\phi$-isoperimetric sets, both when $\phi^*$ is of class $\C^1$ (see Section~\ref{ss:fv-reg}) and when $\phi^*$ is not differentiable (see Section~\ref{ss:fv-cr}). In the former case, we introduce the notion of $\phi$-curvature of a $\C^2$-smooth surfaces (when $\phi^*$ is $\C^2$) and of $\phi$-critical surface. In the latter case, we deduce the (partial) foliation property stated in Theorem~\ref{thmi:folz} for crystalline norms.
In Section~\ref{s:integ} we deduce a foliation property outside the characteristic set for $\phi$-isoperimetric sets of class $\C^2$, assuming $\phi$ and $\phi^*$ to be regular enough. We then study such a foliation from the point of view of geodesics in the ambient space in Section~\ref{s:PMP}, and in Section~\ref{s:char} we study the characteristic set of $\C^2$-smooth {$\phi$-critical surfaces and of} $\phi$-isoperimetric sets, assuming $\phi$ and $\phi^*$ to be $\C^2$ (Theorem~\ref{thmi:char}). In Section~\ref{s:proof} we then prove Theorem~\ref{thmi:class} and we discuss the regularity of the candidate isoperimetric sets $E_\phi$. Finally, Section~\ref{s:cr} is dedicated to general norms and contains the proof of Theorem~\ref{thmi:appr}.

\subsection*{Acknowledgments} The authors thank M.~Ritor\'e and C.~Rosales for pointing out a gap in a preliminary version of the paper. The first and third authors acknowledge the support of ANR-15-CE40-0018 project \textit{SRGI - Sub-Riemannian Geometry and Interactions}.  The first author acknowledges the support received from the European Union's Horizon 2020 research and innovation programme under the \emph{Marie Sklodowska-Curie grant agreement No.~794592}, of the INdAM--GNAMPA project \emph{Problemi isoperimetrici con anisotropie}, and of a public grant of the French National Research Agency (ANR) as part of the Investissement d'avenir program, through the iCODE project funded by the IDEX Paris-Saclay, ANR-11-IDEX-0003-02.

\section{Sub-Finsler perimeter}\label{s:per}

In this section, we introduce the notion of $\phi$-perimeter in $\H^1$ for a norm $\phi$ in $\R^2$. We start by fixing the notation relative to horizontal vector fields and sub-Finsler norms in $\H^1$.

 A smooth horizontal vector field is a vector field $V$ on $\R^3$ that  
can be written as $V=aX+bY$ where $a,b\in \C^\infty(\H^1)$.
When $A\subset \H^1$ is an open set and  $ a,b \in \C^\infty_c(A)$ have compact support in $A$ we shall write $V\in \cD_c(A)$. 
We fix on $\cD(\H^1)$ the scalar product $\langle \cdot ,\cdot\rangle_\cD$ that makes $X,Y$ pointwise orthonormal. Then each fiber $\cD(p)$ can be identified with  the Euclidean plane  $\R^2$.

Let $\phi:\R^2\to[0,\infty)$ be a norm. We fix on $\cD(\H^1)$ the left-invariant norm associated with $\phi$. Namely, with a slight abuse of notation,
 for any $p\in \H^1$ and 
   with $a,b\in \R$
  we define 
\[
\phi(aX(p)+bY(p))  = \phi\big( (a ,b  )\big) . 
\]

Since the Haar measure of $\H^1$ is the Lebesgue measure of $\R^3$, the divergence in $\H^1$ is the standard divergence. Therefore, for  a smooth horizontal vector field $V = a X+ b Y$ 
 we have $\mathrm{div}(V) = Xa + Yb$.

\begin{definition}\label{d:per}
The \emph{$\phi$-perimeter} of a Lebesgue measurable set $E \subset \H^1$ in an open set $A\subset \H^1$ is 
\begin{equation*}
\P_{\phi}(E;A)=\sup \left \{  \int_{E} 
\mathrm{div}(V) 
dp : V 
\in \cD_c(A)  \textrm{ with }
 \max_{\xi\in A} \phi(V (\xi))  \leq 1 \right \}.
\end{equation*} 
When $\P_{\phi }(E;A)<\infty$ we say that $E$ has finite perimeter in $A$.
When $A=\H^1$, we let $ \P_{\phi} (E) = \P_{\phi}(E;\H^1) $.

\end{definition}

Since all the left-invariant norms in the horizontal distribution are equivalent, we have 
$\P_{\phi}(E)<\infty$  if and only if  the set $E$ has finite horizontal perimeter in the sense of  \cite{CDG94,FSSC96,GN96}. 

For regular sets, we can 
represent $\P_\phi(E)$ 
integrating on $\partial E$ a kernel related to the normal. Let $\nu_E$ be the Euclidean unit 
 inner normal 
 to $\partial E$.  We define the horizontal vector field  $N_E:\partial E\to \cD(\H^1)$ by 
\begin{equation*} \label{N_E}
N_E=\langle \nu_E ,X \rangle X+\langle \nu_E , Y \rangle Y,
\end{equation*}
where $\langle\cdot,\cdot\rangle$ denotes the Euclidean scalar product in $\R^3$.

\begin{proposition}[Representation formula]\label{p:repr}
Let $E\subset \H^1$ be a  set with Lipschitz boundary. 
Then for every open set $A\subset \H^1$  we have
\begin{equation}\label{eq:repr}
\P_\phi
(E;A) = \int_{\partial E\cap A} \phi^* ( N_E) \; d\cH^{2},
\end{equation}
where   $\cH^{2}$ is the standard $2$-Hausdorff measure in $\R^3$.
\end{proposition}

\begin{proof} Let  $V \in \cD_c (A)$  be such that $\phi( V) \leq 1$.
By the standard divergence theorem and by the definition of dual norm, we have 
\begin{equation*}
\begin{split}
\int_{E}  \mathrm{div} (V)   \,d\xi
& 
= -\int_{\partial E} \langle V, N_E\rangle_\cD  \; d\cH^2
\leq \int_{\partial E\cap A} \phi(V) \phi ^*   (N_E)  d\cH^{2}
\\
&
\leq \int_{\partial E\cap A } \phi^*(N_E)  d\cH^{2}.
\end{split}
\end{equation*}
By taking the supremum over all admissible $V $ we then obtain
\begin{equation*}
\P_\phi (E;A) \leq  \int_{\partial E\cap A} \phi^*(  N_E)  
d\cH^{2}.
\end{equation*}
 
To get the opposite inequality it is sufficient to prove that for every $\varepsilon >0$ there exists $V \in \cD _c (A)  $ such that $\phi( V ) \leq 1$
and 
	\begin{equation*}\label{eq:Rappr1}
-\int_{\partial E} \langle V, N_E\rangle_\cD  \; d\cH^2 \geq \int_{\partial E\cap A} \phi^* ( N_E)  d\cH^{2} - \varepsilon.
	\end{equation*}
	Here, without loss of generality, we assume that $A$ is bounded. We will construct such a $V$ with continuous coefficients  and with compact support in $A$. The smooth case $V\in \cD_c(A)$ will follow by a standard regularization argument.
	 
	Let us  define the sets 
	\begin{equation*}
		\begin{split}
			\mathcal{U} &= \big \{ p  \in \partial E \cap A : N_E(p )  \textrm{ is defined} \big \},\quad 
			\mathcal{Z} = \big \{ p   \in \mathcal{U} : N_E (p) =0 \big \}.
		\end{split}
	\end{equation*}
From the results of \cite{B03} it follows that 
$\mathcal Z$ has vanishing $\cH^2$-measure. 
For any $p \in 
 \mathcal{U} \setminus \mathcal{Z} $ we take $ V \in \cD(p)$ such that $\phi(V)=1 $ and
\[
  \langle  V, N_E\rangle_\cD  = \phi^* ( N_E).
 \] 
In general, this choice is not unique. However, there is a selection $p\mapsto  V (p)$ that  is measurable (this follows since the coordinates are measurable, see 
for instance \cite[Theorem~8.1.3]{AubinFrank}). 
 We extend $ V$ to $\mathcal Z$ letting $ V =0$ here. This extension is still measurable.

	   Since $\partial E \cap A$ has
	finite $\cH^{2}$-measure, by Lusin's theorem there exists a compact set $K_\varepsilon \subset \partial E \cap A$ 
	such that $\cH^{2}\big( (\partial E \cap A) \setminus K_\varepsilon \big) < \varepsilon$ and the restriction of $ V $ to $K_\varepsilon$ 
	is continuous.
	Now, by  Tietze--Uryshon theorem we extend $V$ from $K_\varepsilon$ to $A$ in such a way that 
	the extended map, still denoted by $V$, is continuous with   compact support in $A$ and satisfies $\phi (V) \leq 1$ everywhere.

	Our construction yields the following
	\begin{equation*}
		\begin{split}
			\int_{\partial E \cap A} \phi^* ( N_E)  d\cH^{2} &= \int_{K_\varepsilon} \langle  V, N_E\rangle_\cD  \, d\cH^{2} 
							+ \int_{(\partial E \cap A)\setminus K_\varepsilon} \phi^* ( N_E )  d\cH^{2} 
							\\
				&= \int_{\partial E \cap A } \langle  V, N_E\rangle_\cD  
				\, d\cH^{2} - 
							\int_{(\partial E \cap A)\setminus K_\varepsilon} \left( \langle  V, N_E\rangle_\cD    - \phi^* (N_E)
							\right) d\cH^{2} 
							 \\
				&\leq \int_{\partial E \cap A} \langle  V, N_E\rangle_\cD   \, d\cH^{2} + C\varepsilon.
		\end{split}
	\end{equation*}
In the last inequality we used the fact that $\langle  V, N_E\rangle_\cD    - \phi^*(  N_E) $ is bounded and 
	$\cH^{2}\big( (\partial E \cap A) \setminus K_\varepsilon \big) < \varepsilon$.
The claim follows.
\end{proof}

 \section{Existence of isoperimetric sets} \label{s:ex}
For a measurable set $E\subset \H^1$ with positive and finite measure 
and a given norm $\phi$ on $\R^2$
we define the  $\phi$-\emph{isoperimetric quotient} as
\[
\operatorname{Isop_{\phi }}(E)=\dfrac{\P_{\phi}(E)}{\mathcal L^3(E)^{\frac{3}{4}}},
\]
where $\mathcal L^3$ denotes the Lebesgue measure of $\R^3$.

 The isoperimetric quotient is invariant under left-translation (w.r.t.\ the operation in \eqref{eq:*Heis}), i.e., $\operatorname{Isop}_{\phi}(p*E)=\operatorname{Isop}_{\phi}(E)$ for any $p\in\H^1$ and $E\subset\H^1$ admissible, and it is $0$-homogeneous with respect to the one-parameter family of automorphisms \eqref{eq:dilat},  i.e.,  $\operatorname{Isop}_{\phi}(\lambda E )=\operatorname{Isop}_{\phi}(E)$, 
 where $\lambda E = \delta_\lambda(E)$.  

The isoperimetric problem \eqref{eq:ip} is then equivalent to minimizing the isoperimetric quotient among all admissible sets. 
Namely,  given $m\in(0,\infty)$, any isoperimetric set $E\subset\H^1$ with $\mathcal L^3(E)=m$ is a solution to 
\begin{equation}\label{eq:isop1}
C_I=\inf\left\{\operatorname{Isop}_{\phi}(E) : E\subset \H^1\text{ measurable, }0<\mathcal L^3(E)<\infty\right\},
\end{equation}
and, \emph{vice versa}, any  solution $E\subset\H^1$ to \eqref{eq:isop1} solves \eqref{eq:ip} within its volume class, i.e., with $m=\mathcal L^3(E)$.
In particular, we have 
\begin{equation}\label{eq:cisop}
C_I=\inf\left\{\P_{\phi}(E) : E\subset \H^1\text{ measurable, }\mathcal L^3(E)=1\right\}.
\end{equation}
The constant $C_I$ depends on $\phi$.

Since $\P_\phi$ is equivalent to the standard  horizontal   
perimeter, the   isoperimetric inequality in \cite{GN96} implies that $C_I>0$ and the validity of the following inequality for any measurable set $E$ with finite measure:
\begin{equation}\label{eq:iisop}
\P_{\phi}(E)\geq C_I \mathcal L^3(E)^{\frac{3}{4}}.
\end{equation}
The constant $C_I$ is the largest one making true the above inequality and isoperimetric sets are precisely those for which \eqref{eq:iisop} is an equality.

\begin{theorem}[Existence of isoperimetric sets]\label{thm:ex}  
Let $\phi$ be any norm on $\R^2$. There exists a set $E\subset\H^1$ with 
non-zero and finite $\phi$-perimeter such that  
\begin{equation}\label{eq:ex}
\P_\phi(E)=C_I \mathcal L^3(E)^{\frac{3}{4}}.
\end{equation}
\end{theorem}

Theorem~\ref{thm:ex} follows by applying the strategy of \cite[Section 4]{LR03}. 
We provide a sketch of the proof only for the sake of completeness. 
In the sequel we denote the left-invariant homogeneous ball centered at $p\in\H^1$ of radius $r>0$ by $B(p,r)$.   

\begin{proof}[Proof of Theorem~\ref{thm:ex}]
By perimeter and volume homogeneity with respect to $\{\delta_\lambda\}_{\lambda\in\R}$ it is enough to prove the existence of a minimizing set in the class of volume   $\mathcal L^3(E)=1$. Let $\{E_k\}_{k\in\N}$ be a minimizing sequence for \eqref{eq:cisop} such that for $k\in\N$ we have
\[
\mathcal L^3(E_k)=1,\qquad \P_{\phi}(E_k)\leq C_I\left(1+\frac{1}{k}\right).
\]

Assume that  there exists $m_0\in(0,1/2)$ such that for any $k\in\N$ there exists $p_k\in\H^1$ satisfying
\begin{equation}\label{eq:ass}
\mathcal L^3(E_k\cap B(p_k,1))\geq m_0.
\end{equation} 
Then, the translated sequence 
$\{-p_k*E_k\}_{k\in\N}$,  still denoted $\{E_k\}_{k\in\N}$,  is also minimizing for \eqref{eq:cisop} and satisfies 
$\mathcal L^3(  E_k \cap B(0,1)) \geq m_0$.

Since $\P_\phi$ is equivalent to the standar horizontal perimeter,
we have a compactness theorem for sets of finite $\phi$-perimeter as in \cite[Theorem~1.28]{GN96}. Then, we can extract a sub-sequence, still denoted $\{ E_k\}_{k\in\N}$, converging in the $L^1_{\mathrm{loc}}(\H^1)$ sense to a set $E\subset\H^1$ of finite $\phi$-perimeter. The lower semi-continuity of $\P_\phi$ therefore implies 
\begin{equation*}\label{eq:existence4}
	\P_{\phi}(E) \leq \liminf_{k\to \infty} \P_{\phi}( E_k) \leq C_I.
\end{equation*}
Moreover, we have
\begin{equation}\label{eq:existence5}
\begin{split}
& 
\mathcal L^3(E) \leq \liminf_{k\to \infty} \mathcal L^3( E_k)=1 \quad \text{  and  }
\\
&
\mathcal L^3(E \cap B(0,1) ) = \lim_{k\to \infty} \mathcal L^3(  E_k \cap B(0,1)) \geq 
m_0. 
\end{split}
\end{equation}

To prove \eqref{eq:ex} we are left to show that $\mathcal L^3(E)=1$, which  follows by applying a sub-Finsler version of \cite[Lemma~4.2]{LR03}, ensuring existence of a radius $R>0$ such that $\mathcal L^3(E\cap B(0,R))=1$. This is based on \eqref{eq:existence5} and on a canonical relation between perimeter and derivative of volume  in balls with respect to the radius,   
which is valid in quite general metric structures, including sub-Finsler ones, see \cite[Lemma~3.5]{Amb01}.

We conclude by justifying the assumption \eqref{eq:ass}. This follows by a sub-Finsler version of \cite[Lemma~4.1]{LR03}. Using once more the equivalence of $\P_{\phi}$  with the standard horizontal perimeter,  
we deduce from \cite[Theorem~1.18]{GN96} the validity of the following \emph{relative isoperimetric inequality} holding for a constant $C>0$ and any measurable set $E$ with finite measure 
\begin{equation*}\label{eq:riisop}
	\min \left\{  \mathcal L^3(B\cap E)^{\frac{3}{4}}, \mathcal L^3(B\setminus E)^\frac{3}{4} \right\} \leq C \P_{\phi}(E, \lambda B),
\end{equation*}
where $\lambda\geq 1$ is  
a constant  depending only on  $\phi$, and $B$ is any left-invariant homogeneous ball.
Together with the fact that the family $\{B(p,\lambda) : p\in\H^1\}$ has bounded overlap, we can reproduce the argument of \cite[Lemma~4.1]{LR03} and prove the claim.
\end{proof}

\begin{remark}\label{rem:bdd}
Following the arguments of \cite[Lemma~4.2]{LR03},  one also shows that any isoperimetric set is equivalent to a bounded and connected one (i.e., it is bounded and connected up to sets of zero Lebesgue measure). 
\end{remark}

\section{First variation of the isoperimetric quotient}\label{s:fv}

In this section we derive  a first order 
necessary condition for $\phi$-isoperimetric sets, both when $\phi$ is regular or not. 

\subsection{Notation} We now introduce some notation that will be used throughout the paper.  

       Let $E,A\subset \H^1$ be sets such that   $E$ is closed, $A$ is open  and there exists a  function $g\in \C^1( A)$, called  \emph{defining function for $\partial E\cap A$}, such that $\partial E\cap A=\{p\in A : g(p)=0\}$ and $\nabla g(p)\ne 0$ for every $p\in \partial E\cap A$. 
  We say that $ E\cap A$ is a \emph{$z$-subgraph} if there exist an open set $D\subset \R^2$ and a function $ f\in  \C^1 (D)$, called 
   \emph{graph function} for $\partial E\cap A$,
   such that 
  \[
    E\cap A  = \{ (\xi, z)\in A : \xi \in D \textrm{ and } z\leq f(\xi) \}.
    \]
In this case, $g(\xi,z)=f(\xi)-z$ is a defining function for  $\partial E\cap A$.

The definition of \emph{$z$-epigraph} is analogous and all results given below for $z$-subgraphs have a straightforward counterpart for 
$z$-epigraphs. 
In a similar way, 
one can also define \emph{$x$-subgraphs}, \emph{$y$-subgraphs}, \emph{$x$-epigraphs}, and \emph{$y$-epigraphs}.

Given a  function $g\in \C^1(A)  $,   
 we denote by $\mathcal G=(Xg) X+(Yg)Y
 $ the \emph{horizontal gradient} of $g$ and
 we define
 the \emph{projected horizontal gradient} as 
 \begin{equation}\label{egge}
 G=(Xg,Yg)
 \in   \R^2.
\end{equation} 
If $\partial E\cap A$ is a $z$-graph with graph function $f\in\C^1 (D)$, we define $F: D\to \R^2$ by
 \begin{equation}\label{effe}
 F(\xi)=G(\xi,f(\xi))=\nabla f(\xi)-\frac12 \xi^\perp,
\end{equation} 
and
\begin{equation}
\label{eq:charf}
\cC(f) =
\{ \xi\in D : F(\xi)=0\}.
\end{equation}
Hence $\cC(E)\cap A=\{(\xi,f(\xi)):\xi\in \cC(f)\}$, 
where $\cC(E)$ is the characteristic set of $E$, defined in \eqref{eq:char}. The set $\cC(f)$ has zero Lebesgue measure in $D$.

If $E\cap A$ is the   $z$-subgraph of a function $f\in\C^1 (D)$, 
by the representation formula \eqref{eq:repr}  
we have
 \[
 \P_\phi(E;A) 
 = \int _D \phi^*(F (\xi)  ) d\xi.
 \]

 When the dual norm $\phi^*$ is of class $\C^1$, starting from a graph function $f\in\C^1 (D)$  we define the vector field $\X: D\to \R^2$ 
by
  \begin{equation}\label{eq:Xf}
    \X (\xi ) = \nabla \n (F(\xi )),\quad 
    \xi \in D.
     \end{equation} 
  The geometric meaning of the  vector field $\X$  will be clarified in the next section, see \eqref{eq:cN}.
  
 \begin{remark} 
At any point $\xi\in D$ such that $F(\xi) \neq 0$ 
  the vector field $\X$  satisfies
\begin{equation}\label{l:norm1}
   \phi(\X(\xi)) =1,
\end{equation}
since the gradient of $\n$ at any nonzero point has norm $\phi$ equal to one (even when $\n$ is not regular, by replacing the gradient by any element of the subgradient; see, for instance, 
\cite[Theorem~1.7.4]{Schneider93}.
\end{remark}

\subsection{Regular norms}\label{ss:fv-reg}

\begin{proposition}[First variation for isoperimetric sets]
\label{pix}
  Let $\phi$ be a norm such that $\phi^*$ is of class $\C^1$.
  Let $E\subset\H^1$ be a $\phi$-isoperimetric   set such that,
for some open set  $A\subset \H^1$,
  $E\cap A$ is a  $z$-subgraph of class $\C^1$. Then the graph function $f\in \C^1(D)$ satisfies in the weak sense the partial differential equation
\begin{equation}\label{EQU}
 \mathrm{div}\big (\X\big) =   -
 \frac{3}{4} \frac{\P_{\phi}(E)}{\mathcal {L} ^3(E)} 
 \quad \textrm{in } D.
\end{equation}

\end{proposition}

\begin{proof}
For small $\varepsilon\in\R$ and 
$\varphi\in \C^\infty_c(D)$ let 
$E_\varepsilon\subset \H^1$ be the set such that 
\[
E_\varepsilon \cap A 
=\{(\xi ,z)\in A :  z\leq f(\xi)+\varepsilon\varphi(\xi ),\ \xi \in D\},
\]
and $E_\varepsilon \setminus A 
=E\setminus A 
$.
Starting from the representation formula 
\begin{equation}\label{eq:svi1}
\P_{\phi}(E_\varepsilon;A) = \int_{\partial E_\varepsilon \cap A} \phi^*(   N _{E_\varepsilon})  d \cH^2 
= \int_D  \phi^* (F
+\varepsilon (X\varphi,Y\varphi)
) d\xi,
\end{equation}
we compute the derivative
\begin{equation}\label{eq:der-per}
{\P_\phi}' = \left. 
\frac{d}{d\varepsilon}\P_{\phi}(E_\varepsilon;A) \right|_{\varepsilon=0}
= \int_D \langle  \X 
, (X\varphi,Y\varphi)
\rangle d\xi
=\int_D \langle  \X 
, \nabla \varphi
\rangle d\xi.  
\end{equation}
  On the other hand, the derivative of the volume is  
\[
\left. 
\mathcal {V} ' = \frac{d}{d\varepsilon}\mathcal {L}^3  (E_\varepsilon ) \right|_{\varepsilon=0}
= \int_D  \varphi\, d\xi.  
\]
 Inserting these formulas 
 into
 \[
  0 = 
\left. 
\frac{d}{d\varepsilon} \frac{ \P_{\phi}(E_\varepsilon
)^{4}}{\mathcal {L}^3(E_\varepsilon
)^{3}}  \right|_{\varepsilon=0}
= 
\frac
{\P_{\phi}(E)^{3}}{\mathcal {L}^3(E 
)^{4}}\big( 4 {\P_{\phi}}' 
\mathcal {L}^3(E 
)-3 \mathcal {V}' \P_{\phi}(E )
\big),
\]
 we obtain
\[
 \int_D 
 \langle  \X
 , \nabla \varphi
 \rangle d\xi
 = \frac 34 
 \frac{\P_{\phi}(E)}
 {\mathcal {L}^3(E 
)}
\int_D  \varphi\, d\xi
\]
  for any test function  $\varphi\in \C^\infty_c(D)$.
This is our claim.  
\end{proof}

{
Proposition \ref{pix} still holds if we only have $f\in\mathrm{Lip}
(D)$.
} 
If   $\phi^*$ is of class $\C^2$ and   $f \in \C^2(D)$ then we have  $\X \in \C^1(D\setminus \cC(f) ;\R^2)$.
So  equation \eqref{EQU}
is satisfied pointwise in $D\setminus \cC(f) $ in the strong sense.

\begin{definition} \label{def:zphicurv} Let $f\in \C^2(D)$ and $\mathcal Xf$ be defined as in \eqref{eq:Xf}.
We call the function $H_\phi  :  D\setminus \cC(f)\to\R$
\begin{equation}
\label{eq:curvz}
H_{\phi}(\xi) =  \mathrm{div}\big (\X (\xi) \big) ,\quad \xi \in   D\setminus \cC(f),
\end{equation}
the \emph{$\phi$-curvature of the graph $\mathrm{gr}(f)$}. 
{We say that $\mathrm{gr}(f)$ has \emph{constant $\phi$-curvature} if there exists $h\in\R$ such that $H_\phi=h$ on $D\setminus\cC(f)$.
Finally, we say that $\mathrm{gr}(f)$ is \emph{$\phi$-critical} if there exists $h\in\R$ such that
 \begin{equation}\label{eq:daverif}
 \int_D 
 \langle  \X
 , \nabla \varphi
 \rangle d\xi
 = - h 
\int_D  \varphi\, d\xi
 \end{equation}
 is satisfied for every $\varphi\in C^\infty_c(D)$.
}
\end{definition}

 Proposition \ref{pix} then asserts that the part of the  boundary of a $\phi$-isoperimetric set of class $\C^2$ that can be represented as a $z$-graph {is $\phi$-critical and in particular} 
it  has constant $\phi$-curvature at noncharacteristic points.

 {

\begin{remark}\label{rem:x-graphs}
Let us discuss how the proof of Proposition~\ref{pix} can be 
adapted to the case where $E\cap A$ is   a $x$-subgraph of class $\C^2$. The case of  $y$-subgraphs is analogous.  
We have a defining function for  $\partial E\cap A$  of the type $g(x,y,z)=f(y,z)-x$ with $f\in \C^2  (D)$. 
The projected horizontal gradient in \eqref{egge} reads
\[
  G(y,z) = \Big(-1-\frac 1 2  yf_z, f_y +\frac 12 f f_z\Big).
\]
For $\varepsilon \in\R $ and $\varphi\in \C_c^\infty (D)$ let $E_\varepsilon $ be   the $x$-subgraph in $A$ of $f+\varepsilon\varphi$. Then the derivative of the $\phi$-perimeter of $E_\varepsilon$ is  
\begin{equation*}
\begin{split}
\left.
\frac{d}
{d\varepsilon}
\P_{\phi}(E_\varepsilon;A)\right|_{\varepsilon=0} 
 &
 = \int_D  \big \langle \nabla \phi^* (G) , \big( -y \varphi_z /2 ,    \varphi _y +(\varphi f)_z /2 \big) \big \rangle dydz 
 \\
 & = - \int_D   \varphi (y,z)\,  \mathcal L f(y,z) \, dydz ,
\end{split}
\end{equation*}
where $\mathcal L : \C^2(D)\to \C(D)$ is the partial differential operator
\begin{equation}
\label{eq:Delta}
  \mathcal L f =  \Big( \frac {\partial }{\partial y}+ 
  \frac {f}{2}
  \frac {\partial }{\partial z}\Big) {\phi}^*_b(G)-\frac y  2 \frac {\partial }{\partial z}{\phi}^*_a(G),
\end{equation} 
{with} 
$\nabla {\phi}^* = ({\phi}_a^*,{\phi}_b^*)$.
 
The statement for $x$-graphs is then that if $E\subset\H^1$ is $\phi$-isoperimetric and 
$E\cap A$ is a  $x$-subgraph with graph function $f\in \C^2 (D)$, then 
\[
\mathcal L f = -
 \frac{3}{4} \frac{\P_{\phi}(E)}{\mathcal {L} ^3(E)}\quad \textrm
 {in }D.
\]
When we only have $f\in\mathrm
{Lip }(D)$, $\mathcal L f $ is well-defined in the distributional sense. 

\end{remark}

 }

\subsection{Crystalline norms} \label{ss:fv-cr}

In this section  we focus on a norm  $\phi$ having non-differentia\-bility points, and in particular on the case where it is
crystalline.
Recall that the dual norm $\phi^*$ to a non-differentiable one is not strictly convex, so that $\nabla \phi^*$ is constant on subsets of $\R^2$ having nonempty interior.

{
\begin{lemma}\label{lem:foliazione1}
Let   $O$ be a 
subset of $\R^2$ where $\nabla \phi^*$ exists and is constant.
Let $E\subset\H^1$ be
such that   $E\cap A
=\{(\xi,z)\in A :  z\leq f(\xi),\ \xi\in D\}$
for some open set  $A\subset \H^1$
and
$f\in \mathrm{Lip}(D)$.
If  $F(\xi)\in O$
 for almost
 every $\xi\in D$ then $E$ is not $\phi$-isoperimetric.
\end{lemma}

}

\begin{proof}
As in the proof of Proposition~\ref{pix}, consider
$\varphi\in \C^\infty_c(D)$ and, for $\varepsilon\in\R$ small, let
$E_\varepsilon\subset \H^1$ be the set such that
\[
E_\varepsilon \cap A
=\{(\xi,z)\in A :  z\leq f(\xi)+\varepsilon\varphi(\xi),\ \xi\in D\},
\]
and $E_\varepsilon \setminus A
=E\setminus A
$.
Then, as in  \eqref{eq:der-per},
\[
{\P_{\phi}}' =
\left.
\frac{d}{d\varepsilon}\P_{\phi}(E_\varepsilon;A) \right|_{\varepsilon=0}
= \int_D \langle {\nabla\phi^*(F)}, \nabla \varphi\rangle d\xi.
\]

\noindent
By hypothesis,  {$\nabla\phi^*(F)$} 
is constant on $D$, so that  ${\P_{\phi}}' =0$.

Now, choosing   $\varphi\neq 0$ with constant sign, we deduce that
\[\left.
\frac{d}{d\varepsilon} \frac{ \P_{\phi}(E_\varepsilon
)^{4}}{\mathcal {L}^3(E_\varepsilon
)^{3}}  \right|_{\varepsilon=0}
=
{
-3\frac
{\P_{\phi}(E)^{4}}{\mathcal {L}^3(E
)^{4}}  \int_D   \varphi(\xi)d\xi
\ne 0,
}
\]
contradicting the extremality of $E$ for the isoperimetric quotient.
\end{proof}

We are ready for the proof of   Theorem \ref{thmi:folz}.
This theorem disproofs Conjecture~8.0.1 in \cite{Sanchez-thesis}, where Pansu's bubble was conjectured to solve the isoperimetric problem for crystalline norms.

Let $\phi$ be a crystalline norm and denote by $v_1,\ldots, v_{2N}\in \R^2 $ the ordered vertices of the polygon $C_\phi=C_\phi(0,1)$. 
Notice that $v_{i+N}=-v_{i}$ for $i=1,\dots,N$.
The dual norm   $\phi^*$ is also crystalline and the vertices of $C_{\phi^*}(0,1)$
are in one-to-one correspondence with the edges $e_i = v_i-v_{i-1}$ of 
$C_\phi(0,1)$ (with $v_0=v_{2N}$). Namely,
$C_{\phi^*}(0,1)$
is the convex hull of $v_1^*,\ldots, v_{2N}^*$  where, for $i=1,\dots,2N$,  the vertex $v_i^*$ is the unique vector  of $\R^2$ 
such that 
\begin{equation}
\label{ORTO}
\langle v_i^*, e_i\rangle =0
\end{equation}
and
$\langle v_i^*,v_i\rangle =
\langle v_i^*,v_{i-1} \rangle=1$. 
In particular, $v_{i+N}^*=-v_{i}^*$ for $i=1,\dots,N$.

 Along the lines $L_i  =\R v_i^*$, the norm $\phi^*$ is not differentiable.
In the positive convex
cone bounded by $\R^+ v_i^*$ and $\R^+ v_{i+1}^*$ the gradient $\nabla \phi^*$ exists and is constant, and we have $\nabla \phi^* = v_i$. 
For piecewise $\C^1$-smooth $\phi$-isoperimetric sets
the projected horizontal gradient $F$
takes values in $L_1\cup\ldots\cup L_{N}$, by Lemma~\ref{lem:foliazione1}.

\begin{proof}[Proof of Theorem~\ref{thmi:folz}]
Let $f\in \C^2(D)$ be  the graph function of $\partial E\cap A$.
For $i=1,\ldots, N$, we let
\begin{equation*}\label{eq:SXY}
D_i=\{ \xi \in D:
F(\xi)\in
{L_i =
\R v_i^*
}
\}.
\end{equation*}
If $\xi\in D_i$ then by \eqref{ORTO} we have
\[
F(\xi)^\perp \in \R(  v_i^*)^\perp = \R e_i.
\]
This implies that
the vector field $X_i$  in \eqref{eq:X_i}
is tangent to $\partial E\cap A$ at the point $(\xi,f(\xi))$.

We are going to prove the theorem by showing that $D=D_i$ for some~$i\in \{1,\dots,N\}$.
Notice that, for $i,j\in \{1,\dots,N\}$ and $i\ne j$, 
$v_i$ and $v_j$ are linearly independent.
By Lemma~\ref{lem:foliazione1} we have that $D=\cup_{i=1}^N D_i$.
We claim, moreover, that
\begin{equation}\label{eq:cover}
\overline{D}= \cup_{i=1}^{N} \overline{\operatorname{int}D_i}.
\end{equation}
In order to check the claim, pick $\xi\in D$ and assume by contradiction that
$\xi\not\in \overline{\operatorname{int}D_i}$ for $i=1,\dots,N$.
Let $i_1$ be such that $\xi\in D_{i_1}$.
Since $\xi\not\in \operatorname{int}D_{i_1}$, for every $\epsilon>0$ the set $D\setminus D_{i_1}$ intersects the disc of radius $\epsilon$ centered at $\xi$.
Hence, there exist  
$i_2\ne i_1$, and a sequence $(\xi_n)_{n\in\N}$ in $D_{i_2}\setminus D_{i_1}$ converging to $\xi$.
Now, either $\xi_n\in \operatorname{int}D_{i_2}$ for infinitely many $n$ or $\xi_n\not\in \operatorname{int}D_{i_2}$ for $n$ large enough. In the first case $\xi\in \overline{\operatorname{int}D_{i_2}}$, leading to a contradiction. In the second case, we repeat the reasoning leading to $(\xi_n)_{n\in\N}$, replacing $D_{i_1}$ by $D_{i_2}$ and $\xi$ by $\xi_n$ for every $n\in \N$, and, by a diagonal argument, we obtain 
$i_3\ne i_1,i_2$, and a sequence $(\hat \xi_n)_{n\in\N}$ in $D_{i_3}\setminus(D_{i_1}\cup D_{i_2})$ converging to $\xi$.
Repeating the argument finitely many times,
we end up with $i_{N}\in \{1,\dots,N\}$ and a sequence $(\tilde \xi_n)_{n\in\N}$ in $D_{i_N}\setminus(\cup_{j=1}^{N-1}D_{i_j})$ converging to $\xi$ with $D=D_{i_1}\cup \dots \cup D_{i_{N}}$. Since $D_{i_N}\setminus(\cup_{j=1}^{N-1}D_{i_j})=D\setminus(\cup_{j=1}^{N-1}D_{i_j})$ is open, we deduce that $\xi\in  \overline{\operatorname{int}D_{i_N}}$. This concludes the contradiction argument, proving \eqref{eq:cover}.

Let $v_i$ and  $v_j$ be linearly independent.
We claim that
\begin{equation}\label{eq:claimint}
\overline{\operatorname{int}(D_i)}\cap \overline{\operatorname{int}(D_j)}
=\emptyset.
\end{equation}

Consider the vector field $X'$ on $D\times \R$ defined by $X'(\xi,z)=(e_i,e_{i,1}f_x(\xi)+e_{i,2}f_y(\xi))$. Then $X'$ is $\C^1$ and both $X'$ and $X_j$ are tangent to $\partial E\cap A$ in a neighborhood of any point of $S_j$, where
\[S_k=\{(\xi,f(\xi)): \xi\in
\operatorname{int} (D_k)
\},\quad k=1,\dots,N.
\]
Hence $[X',X_j]\in T_\xi(\partial E\cap A)$ for every $\xi\in S_j$.
On the other hand,  $X'$
coincides with $X_i$ on  $S_i\times \R$, and therefore $[X',X_j]=[X_i,X_j]=c_{ij} Z$ on $S_i$, with $c_{ij}\in \R\setminus\{0\}$.
Assume by contradiction that $\overline{S}_i\cap \overline{S}_j$ contains at least one point $\xi$.
By continuity of $[X',X_j]$, we deduce from the above reasoning that $Z(\xi)\in T_\xi(\partial E\cap A)$.
The contradiction comes from the remark that, by definition of $S_i$ and $S_j$, also $X_i(\xi)$ and $X_j(\xi)$ are in $T_\xi(\partial E\cap A)$.
We proved \eqref{eq:claimint}.

We deduce from \eqref{eq:cover} and \eqref{eq:claimint} that $\{\overline{\operatorname{int}(D_1)},\dots, \overline{\operatorname{int}(D_N)}\}$ is an open disjoint cover of $\overline{D}$.
We conclude by connectedness of $D$.
\end{proof}

\section{Integration of the curvature equation}\label{s:integ}

Throughout this section  $\n$ is a norm of class   $\C^2$, unless explicitly mentioned other\-wise.

Let $A\subset \H^1$ be open and $g\in \C^{1,1}(A)$ be such that
$\nabla g(p)\ne 0$ for every $p$ in
$\Sigma=\{p\in A : g(p)=0\}$. 
The projected horizontal gradient $G:A\to\R^2$ introduced in \eqref{egge}
is Lipschitz continuous.
Assume that $\Sigma$ has no characteristic points, that is, $G(p)\ne 0$ for every $p\in \Sigma$. 
We use the coordinates  $G=(a,b)$ with $a,b\in \mathrm{Lip}(A)$ and we consider 
  $G^\perp=(-b,a)$.
  The horizontal vector field  $\mathcal G ^\perp =-bX+aY$  is tangent to $\Sigma$.

\begin{definition}\label{def:Legendre}
 A curve $\gamma\in \C^1( I;\Sigma) $ is said to be a \emph{Legendre curve} of $\Sigma$ if $\dot\gamma(t) = \mathcal G^\perp (\gamma(t))$ for all $t\in I$. 
\end{definition}

\noindent
In coordinates, a curve $\gamma=(\xi,z)$ in $\Sigma$ 
is a Legendre curve if and only if  
\begin{equation*}
\dot \xi =G^\perp(\gamma)\quad\textrm{and}\quad 
\dot z=\omega (\xi, \dot\xi)  .
\end{equation*}
Since $\mathcal G^\perp$ is Lipschitz continuous, the graph $\Sigma$ is foliated by Legendre curves: 
for any $p\in \Sigma$ there exists a unique (maximal) Legendre curve passing through $p$.

Consider now the case where $\Sigma$ is a $z$-graph with graph function 
$f\in \C^{1,1}(D)$, where $D$ is an open subset of $\R^2$. 
Then $G(\xi,f(\xi))=F(\xi)$, where $F$ is defined as in  \eqref{effe}, and a Legendre curve $\gamma=(\xi,z)$ satisfies
\begin{equation}\label{eq:legendre}
\dot \xi =F^\perp(\xi)\quad\textrm{and}\quad 
\dot z=\omega (\xi, \dot\xi)  .
\end{equation}

The domain $D$ is foliated by integral curves of $F^\perp$.
On $D$ we define the vector field $\mathcal N \in \mathrm{Lip}(D;\R^2)$ by
\begin{equation}\label{eq:cN}
\mathcal N (\xi) =  \X (\xi) = \nabla \n (F(\xi)),\quad \xi\in D.
\end{equation}
We know that $\phi(\mathcal N)=1$, by \eqref{l:norm1}.
We may call $\mathcal N$ the $\phi$-\emph{normal} to the foliation of $D$ by integral curves of $F^\perp$.
We denote by $H_\phi =\mathrm{div}(\cN)$ the divergence of $\mathcal N$.

\begin{theorem}\label{thm:CMCfol}
Let $\phi^*$ be 
of class $\C^2$. Let $\Sigma$ be the $z$-graph of a function  $f\in  \C^{2}(D)$ with $\cC(f)=\emptyset$.
Then any  Legendre curve $\gamma\in \C^1(I;\Sigma)$, with $\gamma = (\xi, z)$,  satisfies  
\begin{equation}\label{eq:dotNperp}
\frac{d}{dt}\cN(\xi)= H_\phi(\xi)  \dot \xi\quad 
\textrm{and}
\quad
\dot z=\omega(\xi,\dot\xi).
\end{equation}
 \end{theorem}

\begin{proof}
The second equality in \eqref{eq:dotNperp} is part of the definition of a Legendre curve.
We prove the first equality.

We identify $\cN(\xi)$ and  
$\dot \xi =F^\perp(\xi )$ with column vectors
and we denote   by $Jg$ the Jacobian matrix of a differentiable mapping  $g$.
By the chain rule, using the coordinates $F =(a,b)$ and  $\dot \xi = ( -b(\xi), a(\xi))$ we obtain
\begin{equation}\label{eq:rperp'}
\begin{split}
&\frac{d}{dt}\cN (\xi )=
{\mathcal H \n}
( F(\xi)  ) JF (\xi) \dot \xi  
=\begin{pmatrix}
-b a_x  \n_{{aa}}
 - b b_x \n_{{ab}}
+a a_y \n_{{aa}}
+a b_y \n_{{ab}}
\\
-b a_x \n_{{ab}}
-b b_x \n_{{bb}}
+a  a_y \n_{{ab}}
+ a  b_y\n_{{bb}}
\end{pmatrix},
\end{split}
\end{equation}
where $\mathcal H\phi^*$ is the Hessian matrix of $\phi^*$ and the second order derivatives of $\phi^*$ are evaluated at $F(\xi)$.
Since $\n$ is of class $\C^2$, we identified $\n_{{ab}}=\n_{{ba}}$.
By Euler's homogeneous function theorem, since $\nabla\n$ is $0$-positively homogeneous 
there holds $\langle \nabla\n_a(F),F\rangle=0$ and $\langle \nabla\n_b(F),F\rangle=0$. These  formulas read
\[
a\n_{{aa}} +b \n_{{ab}}=0\quad\textrm{and}\quad
a \n_{{ab}}+b\n_{{bb}}=0.
\]
Plugging these relations into \eqref{eq:rperp'}, we obtain
\begin{equation}\label{eq:rperp'1}
\frac{d}{dt}\cN(\xi)=
\left(a_x \n_{{aa}}
+b_x \n_{{ab}}
+a_y  \n_{{ab}} 
+b_y  \n_{{bb}} \right)\dot \xi .
\end{equation}

On the other hand, we have 
\[
\operatorname{div}(\cN)= \operatorname{div}(\X)=a_x \n_{{aa}}
+b_x \n_{{ab}}
+a_y \n_{{ab}} 
+b_y \n_{{bb}},
\]
so that \eqref{eq:rperp'1} yields the claim.
\end{proof}

 {

\begin{remark}\label{rem:x-graphs2}
An analogue of Theorem~\ref{thm:CMCfol} holds true 
for 
$x$-graphs.
Let $\Sigma$ be a  $x$-graph $\Sigma$ without  characteristic points
and with defining function   $g(x,y,z)=f(y,z)-x$ for some $f$ of class $\C^2$.
Let $\gamma\in \C^1(I;\Sigma)$ be  a Legendre curve with coordinates $\gamma(t) = ( f(\zeta(t)), \zeta(t))$ for $t\in I$ and consider the vector   $\mathcal N(y,z)=\nabla \n (G(y,z)  )$.
Following the same steps as in the proof of Theorem~\ref{thm:CMCfol}, one gets 
\[
\frac{d}{dt} {\mathcal N}(\zeta )=\mathcal L f (\zeta) G^\perp (\zeta)
\quad \textrm {on }  I.
\]
 Hence, the conclusion of Theorem~\ref{thm:CMCfol} holds with 
 $H_\phi=\mathrm{div}\big (\X\big) $ replaced by the quantity  $\mathcal L f$ defined  in Remark~\ref{rem:x-graphs}.
Notice that $H_\phi$ and $\mathcal L f $ coincide on surfaces that are both $x$-graphs
and $z$-graphs.

An analogous remark can be made for $y$-graphs.

\end{remark}

}

\begin{corollary}\label{l:conto} Let $\phi^*$ be 
of class $\C^2$.
Let $\Sigma$ be the $z$-graph
of a function $f\in \C^2(D)$ with $\cC(f) =\emptyset$.
If $\Sigma$ has constant $\phi$-curvature $ {h}\neq 0$ then it is foliated by Legendre curves
that are horizontal lift of $\phi$-circles in $D$ with radius $1/|{h}|$, 
followed in clockwise  sense if ${h}>0$ and in 
 anti-clockwise sense if  ${h}<0$.
\end{corollary}

\begin{proof} Having constant $\phi$-curvature ${h}$ means that 
\[
  \mathrm{div}(\cN)=  \mathrm{div}(\X) = {h}\quad\textrm{in } D. 
\]
By Theorem \ref{thm:CMCfol}, for any Legendre curve $\gamma =(\xi,z)$ we have  
\[
\frac{d}{dt}\cN(\xi)- H(\xi)  \dot \xi=0.
\]
We may than integrate this equation and deduce that there exists $\xi_0\in\R^2$ such that along $\xi$ we have
 \begin{equation}
 \label{eq:NP}
\cN (\xi  )-{h}\xi =-{h}\xi_0.
\end{equation}
From
 \eqref{l:norm1} and \eqref{eq:cN} we conclude that
\[
|{h}|  \phi (\xi-\xi_0) = \phi( {h}(\xi-\xi_0) ) =\phi(\cN)=1. 
\]
 
Finally,  notice that $\langle \cN(\xi),F(\xi)\rangle>0$ if $F(\xi)\ne 0$, so that $t\mapsto F(\xi(t))$ rotates clockwise if ${h}>0$  and anti-clockwise if ${h}<0$, according to \eqref{eq:dotNperp}. Hence, $t\mapsto F(\xi(t))^\perp$ and $t\mapsto \xi(t)$  also
rotate clockwise if  ${h}>0$, and anti-clockwise if ${h}<0$. 
\end{proof}

{

Let us discuss an extension of Corollary~\ref{l:conto}
to the case
in which we replace the assumption that $\phi^*$ is $\C^2$ by the weaker assumption that $\phi^*$ is \emph{piecewise $\C^2$}, in the following sense: there exist $k\in \N$ and $A_1,\dots,A_k\in \R^2$ such that $\phi^*$ is $\C^2$ on $\R^2\setminus \cup_{j=1}^k {\rm span}(A_j)$. 

A relevant case where this assumption holds true is when  $\phi$ is the $\ell^{p}$ norm 
\[\ell^{p}(x,y)=(|x|^p+|y|^p)^{\frac{1}{p}},\qquad x,y\in\R,\]
with $p>2$. Indeed, the  dual norm $(\ell^{p})^*$ coincides with the norm $\ell^q$, with $q=p/(p-1)<2$, which is 
$\C^2$ out of the coordinate axes, but not
on the whole punctured plane $\R^2\setminus\{0\}$.
We can prove the following.

\begin{corollary}
\label{rem:pieceC2}
Let $\phi^*$ be piecewise $\C^2$.
Let $\Sigma$ be the $z$-graph
of a function $f\in \C^2(D)$ with $\cC(f) =\emptyset$.
If $\Sigma$ has constant $\phi$-curvature ${h}\neq 0$ then it is foliated by Legendre curves
that are horizontal lifts of $\phi$-circles in $D$ with radius $1/|{h}|$,   followed in clockwise sense if ${h}>0$ and
in anti-clockwise sense if   ${h}<0$. 
\end{corollary}
\begin{proof}
Under the assumptions of the corollary, 
the projected horizontal gradient is 
$\C^1$ on $D$ and 
Legendre curves can be introduced as in Definition~\ref{def:Legendre}.
 
Consider any Legendre curve $\gamma=(\xi,z)$ on $\Sigma$. 
Let us denote by $I\subset \R$ the maximal interval of definition of $\gamma$ and by $J$ the open subset of $I$ defined as follows: $t\in J$ if and only if $F(\xi(t))$ is in the region where $\phi^*$ is $\C^2$.
For the restriction of $\gamma$ to a connected component $J_0$ of $J$, 
Theorem~\ref{thm:CMCfol} can be recovered. In particular, 
since 
$\Sigma$ has constant $\phi$-curvature ${h}\ne 0$, then $\gamma|_{J_0}$ is the lift of a $\phi$-circle of radius $1/|{h}|$, followed 
clockwise or anti-clockwise 
depending on the sign of ${h}$.  
 If $t\in I\setminus J$, then  $F(\xi(t))$ belongs to one of the lines ${\rm span}(A_1),\dots,{\rm span}(A_k)$  on which $\phi^*$ may lose the $\C^2$ regularity. 
Notice that the restriction of $\xi$ to a connected component of $J$ compactly contained in $I$ follows an arc of 
 $\phi$-circle connecting two lines of the type 
 ${\rm span}(A_j)$. In particular, it cannot have arbitrarily small length. 

If $I\setminus J$ is made of isolated points,  
then $\gamma:I\to \Sigma$ is the lift of a $\phi$-circle of radius $1/|{h}|$. Indeed, an 
arc of $\phi$-circle of prescribed radius followed in a prescribed sense is only determined by its initial point and its tangent line there.
Since $\gamma$ is an arbitrary Legendre curve on $\Sigma$, the proof is complete if show 
that $I\setminus J$ does not contain intervals of positive length.

Assume by contradiction that $[t_0,t_1]$ is contained in $J$ with $t_0<t_1$. Then $F(\xi(t))$ is constantly equal to some $A\in \R^2$  for $t\in [t_0,t_1]$. 
Let $\delta>0$ and $\kappa:(-\delta,\delta)\to \Sigma$ be a $C^1$ curve such that $\kappa(0)=\gamma(t_0)$ and $\kappa'(0)$ is not proportional to $\gamma'(t_0)$. 
Write $\kappa(s)=(\xi_s,z_s)$ and notice that  $F(\xi_s)$ converges to $A$ as $s\to 0$.
Consider for each $s\in (-\delta,\delta)$ the Legendre curve $\gamma_s$ such that $\gamma_s(t_0)=\kappa(s)$. Then $\gamma_s$  converges to $\gamma$ and $F\circ \gamma_s$  converges to $F\circ \gamma$, uniformly on $[t_0,t_1]$,  as $s\to 0$. 
Hence, for $\varepsilon>0$ and $|s|$ small enough, the restriction 
of $\gamma_s$ to $(t_0+\varepsilon,t_1-\varepsilon)$
cannot contain the lift of any arc of $\phi$-circle of radius $1/|{h}|$. This implies that there exists a nonempty open region of $\Sigma$ 
of the form $\{\gamma_s(t) : t\in(t_0+\varepsilon,t_1-\varepsilon),\;|s|<\bar \delta\}$  
on which $F(\xi)=A$, contradicting 
the assumption that $\Sigma$ has constant nonzero  $\phi$-curvature.
\end{proof}
}

 \section{Foliation property with geodesics}\label{s:PMP}

In this section we prove that the Legendre foliation of a surface (a $z$-graph) with constant $\phi$-curvature  consists of length minimizing curves in the ambient space (geodesics)   relative to the norm  $\phi^\dag$ in $\R^2$  defined by  
\[
\phi ^\dag (\xi)  = \phi^*(\xi^\perp),\quad \xi\in\R^2. 
\]

We consider a general norm $\nB$ in $\R^2$ and we introduce the notation $\{\psi\le1\}=\{\xi\in\R^2 : \psi(\xi)\le 1\}$. For $T\geq0 $ we introduce  the class of curves
\[
\mathcal A_ T=\big\{ \gamma =(\xi,z) \in {\rm AC}([0,T];\H^1) :  \textrm{$\dot z =\omega(\xi,\dot\xi)$   and $\psi( \dot \xi) \leq 1$ a.e.}\big\},
\]
where $\omega$ is the symplectic form introduced in \eqref{eq:omega}.
In the sequel, we denote by 
$u=\dot \xi \in L^1([0,T];\{\psi\le 1\})$
the \emph{control} of $\gamma$.
For given points $p_0,p_1\in \H^1$ we consider the \emph{optimal time problem}
\begin{equation}
\label{eq:TOC} 
   \inf \big \{
    T\geq 0  : 
     \textrm{ there exists $\gamma\in\mathcal A_T$ 
     such that $\gamma(0)=p_0$ and $\gamma(T) =p_1$}
\}. 
\end{equation}
We call a curve $\gamma$ realizing the minimum in \eqref{eq:TOC} a \emph{$\nB$-time minimizer} between $p_0$ and $p_1$. In this case, we call the pair $(\gamma, u)$ with $u=\dot\xi$ an \emph{optimal pair.} A $\psi$-time minimizer is always parameterized by $\nB$-arclength, i.e., $\nB(u) =1$.  So, 
$\nB$-time minimizers are  $\nB$-length minimizers
parameterized by $\nB$-arclength.

An optimal pair $(\gamma,u)$ satisfies the necessary conditions given by
Pontryagin's Maximum Principle.  
As observed in \cite{Ber94}, 
it necessarily is a \emph{normal extremal}, whose definition is recalled below.  
The Hamiltonian associated with the optimal time problem 
\eqref{eq:TOC} is 
${\mathfrak H}:\H^1\times \R^3\times\{\psi\leq 1\}\to\R$
\[
\begin{split}
{\mathfrak H}(p,\lambda,u) & =\left(\lambda_x-\frac{y}{2}\lambda_z\right)u_1+\left(\lambda_y+\frac{x}{2}\lambda_z\right)u_2
 = \langle \lambda_\xi +\frac 12 \lambda _z \xi^\perp, u \rangle,
 \end{split}
\] 
where $\lambda = (\lambda_\xi,\lambda_z)\in\R^2\times\R$.

\begin{definition}
The pair 
$(\gamma, u)\in {\rm AC}([0,T];\H^1)\times 
L^1([0,T];\{\psi\le 1\})$
is a \emph{normal 
extremal} if there exists a nowhere vanishing curve $\lambda \in {\rm AC}([0,T];\R^3)$ such that  $(\gamma,\lambda)$ solves a.e.~the Hamiltonian system 
\begin{equation*}\label{eq:sys:1}
\begin{cases}
\dot \gamma = {\mathfrak H}_\lambda (\gamma,\lambda,u) 
\\
\dot\lambda = - {\mathfrak H} _ p (\gamma, \lambda, u),
\end{cases} 
\end{equation*}
and for every $t\in [0,T]  $ we have
\begin{equation}\label{eq:Hmaxx}
1={\mathfrak H}(\gamma(t),\lambda(t) ,u(t))=\max_{\nB(u)\leq 1}{\mathfrak H}(\gamma(t),\lambda(t),u).
\end{equation}
\end{definition}

In the coordinates $\gamma =(\xi,z)$ and $\lambda =(\lambda_\xi,\lambda_z)$, the Hamiltonian system reads
\begin{equation}\label{eq:sys}
\begin{cases}
\dot \xi =u,\\
\dot z=\omega( \xi, u) ,
\end{cases}
\begin{cases}
\dot \lambda_\xi =\frac 12  \lambda_ z u^\perp ,\\
\dot \lambda_z=0.
\end{cases}
\end{equation}

\begin{theorem}\label{l:PMP}
Let  $\nB$   be of class $\C^1$ and let $\gamma = (\xi,z)  \in {\rm AC}([0,T];\H^1)$ be  a horizontal curve. 
The following statements (i) and (ii) are equivalent: \begin{itemize}
\item[(i)] $\gamma$ is a local $\nB$-length minimizer parametrized by $\nB$-arclength;
\item[(ii)] the pair $(\gamma, u)$ with $u=\dot\xi$ is a normal 
extremal. 
\end{itemize}

\noindent 
Moreover, if $\nB$ is of class $\C^2$ then each of  (i)  and (ii) is equivalent to 
\begin{itemize}
\item[(iii)] $\gamma$ is of class $\C^2$ and parameterized by $\nB$-arclength, and there is  $\lambda_0 \in\R$ such that
\begin{equation}\label{eq:dotN}
\mathcal H \psi( \dot \xi)\ddot \xi  
=\lambda_0 \dot  \xi ^\perp,
\end{equation}
where $\mathcal H \psi$ is the Hessian matrix of $\psi$.
\end{itemize}
\end{theorem}

\begin{proof} The equivalence between (i) and (ii) is \cite[Theorem~1]{Ber94}.
 
Let us show that (ii) implies (iii). We set   
\begin{equation}
\label{eq:rl}
{\mathcal M}(t)=\lambda_\xi(t) +\frac 12 \lambda_z (t)\xi(t)^\perp, \quad t\in [0,T],
\end{equation} 
where $\lambda = (\lambda_\xi,\lambda_z)$ is the curve given by the definition of extremal.
Then the maximality condition in  \eqref{eq:Hmaxx} for normal extremals reads
\begin{equation}\label{eq:Hmax}
1=\langle{\mathcal M}(t),u(t)\rangle=\max_{\nB(u)\leq 1}\left\langle {\mathcal M}(t),u\right\rangle=\nB^*({\mathcal M}(t)).
\end{equation}
This is equivalent to 
the identity
\begin{equation}\label{eq:N}
{\mathcal M}(t)=\nabla\nB(u(t)).
\end{equation}

When $\psi$
is of class $\C^2$, from \eqref{eq:N}, \eqref{eq:rl}, and \eqref{eq:sys}
we obtain the differential equation for $u=\dot\xi$
\begin{equation} \label{EMME}
 \mathcal  H \psi (u)  \dot u = \dot {\mathcal M }= \dot\lambda_\xi +\frac 12 \dot \lambda_z  \xi +\frac12  \lambda _z u^\perp  =\lambda_z u^\perp.
\end{equation}
 This is \eqref{eq:dotN} with $\lambda_0: = \lambda_z$.

 Now we show that (ii) is implied by (iii). Consistently with \eqref{eq:N}, we define $
\mathcal M (t) = \nabla \psi(u(t))$, for $t\in [0,T]$.
Then
 $\nB^*({\mathcal M})=1$.

 We define the curve $\lambda = (\lambda_\xi, \lambda_z)$ letting $\lambda_z =\lambda_0$ and
 $
  \lambda_\xi = \mathcal M -\frac 12 \lambda _z \xi^\perp$.
When $\psi$ is of class $\C^2$, we obtain
\[
 \dot\lambda _\xi =\ \dot{\mathcal M} -\frac 12 \lambda_z \dot\xi^\perp = \mathcal H \psi(\dot \xi)  \ddot\xi -\frac 12 \lambda_z \dot\xi^\perp =\frac 12 \lambda_z u^\perp.
\]
Hence, all equations in \eqref{eq:sys} are satisfied, showing that the pair $(\gamma, u)$ is a normal extremal.
 This proves that (iii) implies (ii).
\end{proof}

\begin{remark} \label{PALLO}
When $\lambda_0\ne 0$,  equation \eqref{eq:dotN} can be integrated in the following way. Using \eqref{EMME}, the equation  is equivalent to
$
  \dot{\mathcal M} = \lambda_0 \dot \xi^\perp ,
$
that implies $\mathcal M = \lambda _0 (\xi^\perp -\xi_0^\perp)$ for some constant $\xi_0\in\R^2$. So from \eqref{eq:Hmax} we deduce that
$
| \lambda_0| \psi^*(\xi^\perp -\xi_0^\perp) = 1.
$
If we choose $\psi = \phi^\dag$ then we have $\psi^*(\xi^\perp) =\phi(\xi)$. So 
the previous equation becomes the equation for a $\phi$-circle
\[
 \phi(\xi-\xi_0) =  {1}/{|\lambda_0|}.
\]

\end{remark}

\begin{corollary}\label{c:isop-lm} Let $\phi$ be a norm with dual norm $\n$  of class  piecewise  $\C^2$ and let $f\in \C^2(D)$  be such that $\cC(f) =\emptyset$.
If  
$\mathrm{gr}(f)$ has constant $\phi$-curvature, 
then it is foliated by geodesics of $\H^1$ relative to the  norm $\phi^\dag$. 
\end{corollary}

The proof is Corollary~\ref{rem:pieceC2},  
combined with Remark \ref{PALLO} and Theorem \ref{l:PMP}.

\section{Characteristic set of $\phi$-critical surfaces}\label{s:char}

In this section we study the characteristic set of  $\phi$-critical surfaces (see Definition~\ref{def:zphicurv}) and then apply the results to $\phi$-isoperimetric sets. For a $\C^2$ surface $\Sigma\subset\H^1$, the characteristic set is
\begin{equation}
\cC(\Sigma)=\{p \in \partial E : T_p \Sigma = \cD(p)\}.
\end{equation}
Note that any $\C^2$ surface $\Sigma\subset\H^1$ is a $z$-graph around any of its characteristic points $p\in\cC(\Sigma)$.

When $\Sigma$ is oriented, the $\phi$-curvature $H_\phi$ of $\Sigma$ can be defined in a globally coherent way. In particular, when  $\Sigma$ is a $z$-graph at the point $p=(\xi, z)=(x,y,z)\in\Sigma$, the $\phi$-curvature at $p\in \Sigma\setminus \cC(\Sigma)$ is defined through \eqref{eq:curvz}, by letting $H_\phi(p)  = \mathrm
{div}(\X)(\xi)$ where $f$ is a $z$-graph function.
When
$\Sigma$ is a $x$-graph, we let $H_\phi(p) = \mathcal L f(y,z)$, where now $f$ is a $x$-graph
function and $\mathcal L f$ is defined in \eqref{eq:Delta}; when $\Sigma$ is a $y$-graph we proceed analogously. 

We say that $\Sigma$ is \emph{$\phi$-critical} if it is closed, has constant $\phi$-curvature and it is $\phi$-critical in the sense of \eqref{eq:daverif} in a neighborhood of any characteristic point.

Our goal is to  prove  Theorem~\ref{thmi:char}. The proof is obtained combining   Lemma~\ref{l:char1} and Theorem~\ref{thm:ccurves} below.

In this section, $\phi$ and $\phi^*$ are two norms of class $\C^2$.
We will omit to mention this assumptions in the various statements.

\subsection{Qualitative structure of the characteristic set}

\begin{lemma}\label{l:char1} 
Let $\Sigma\subset \H^1$ be a $\C^2$ surface with constant $\phi$-curvature. 
Then $\cC(\Sigma)$ consists of isolated points and $\C^1$ curves.
Moreover, for every isolated point $p_0=(\xi_0,z_0)\in \cC(\Sigma)$ and every $f$ such that $p_0\in \mathrm{gr}(f)\subset \Sigma$, 
we have $\operatorname{rank}(JF(\xi_0))=2$, {where $F$ is the
 projected horizontal gradient introduced in \eqref{effe}}.
\end{lemma}

\begin{proof} 
We let $\cC(f)$ be as in \eqref{eq:charf}. For any   $\xi_0 \in \cC(f)$, the Jacobian matrix $JF(\xi_0)$ has rank 1 or 2.  Indeed, an  explicit calculation shows that   $JF(\xi_0)\neq 0$ for all $\xi_0 \in D$. 
 If $\mathrm{rank}(JF(\xi_0))=2$ then $\xi_0$ is an isolated point of $\cC(f)$.

We study the case $\mathrm{rank}(JF(\xi_0))=1$.
We claim that in this case  $\cC(f)$ is a curve of class $\C^1$ in a neighborhood of $\xi_0$. The argument that we use here is inspired by \cite{CHMY}, see also Remark~\ref{rmk:cheng}.

    For $b\in \R^2$ we define 
  $F_b:D\to \R$, 
  $F_b =\langle F,b\rangle$.
      When $b\notin \operatorname{ker}(JF(\xi_0))$, 
      the equation $F_b=0$ defines a $\C^1$ curve $\Gamma_b$ near and through $\xi_0$. We have $\cC(f)\subset \Gamma_b$.
Since 
$\nabla F_b (\xi_0)$ is in the image of $JF(\xi_0)$, which is a line independent of $b$, 
     the normal direction to $\Gamma_b$ at $\xi_0$ does not depend on $b$. We choose one of the two unit  normals and we call it $N\in \R^2$.

  We claim that there exist $a,b\in \mathbb S^1$, where $\mathbb S^1=
   \{w\in\R^2 : |w|=1\}$, such that
\begin{equation}
          \label{nonno}
a\notin \{b,-b\},\quad a,b 
         \notin \operatorname{ker}(JF(\xi_0)),\quad          |\langle \nabla\phi^*(b^\perp ), N\rangle|\neq |\langle \nabla\phi^*(a^\perp), N\rangle|.
  \end{equation} 
  To prove the claim, pick $b\in\mathbb S^1\setminus \operatorname{ker}(JF(\xi_0))$ (this is possible since $\operatorname{rank}(JF(\xi_0))\neq0$), and define the set
\[
K_b:=\left\{v\in C_\phi: |\langle v, N\rangle|=|\langle \nabla\phi^*(b^\perp ), N\rangle|\right\}.
\]
Since the map $\nabla\phi^*:\mathbb S^1\to C_\phi$ is continuous, the set $(\nabla\phi^*)^{-1}(K_b)\subset \mathbb S^1$ is closed in $\mathbb S^1$.
Moreover, $\nabla\phi^*:\mathbb S^1\to C_\phi$ is surjective, 
since for every $w\in C_\phi$ and every $v$ in the subgradient of $\phi^*$ at $w$, we have $w=\nabla\phi^*(v)$ (see, e.g., \cite[Theorem 23.5]{Rockafellar}).
As a consequence, 
$(\nabla\phi^*)^{-1}(K_b)\neq  \mathbb S^1$, since otherwise
we would have $K_b=C_\phi$, which is impossible. The set 
\[\Upsilon=\operatorname{ker}(JF(\xi_0))^\perp\cup(\nabla\phi^*)^{-1}(K_b)\cup\{b^\perp,-b^\perp\}\]
 is therefore a proper closed subset of $\mathbb S^1$, 
and the claim follows by choosing $a^\perp\in\mathbb S^1\setminus \Upsilon$.

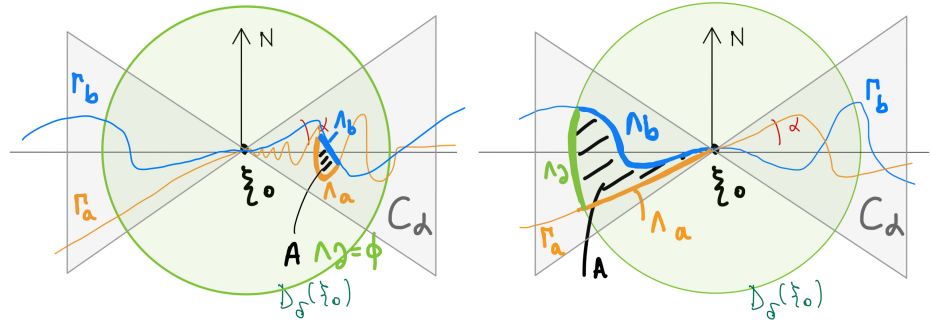
\begin{figure}
\begin{tikzpicture}[smooth cycle,scale=.8]
\filldraw[fill=green!15!white, draw=green!70!black, line width=.75pt]  (0,0) circle (25mm);
\draw (0,0) node[below] {$\xi_0$};
\filldraw[fill=green!15!white, draw=green!70!black, line width=.75pt]  (8,0) circle (25mm);
\draw (8,0) node[below] {$\xi_0$};
\node[green!60!black] at (0,-2) {$\{|\xi-\xi_0|<\delta\}$};
\node[green!60!black] at (8,-2)  {$\{|\xi-\xi_0|<\delta\}$};
\draw[gray] (-3,0) -- (3,0);
\draw[gray] (5,0) -- (11,0);
\draw [-stealth](0,0) -- (0,1.5) node[near end,right]{$N$};
\draw [-stealth](8,0) -- (8,1.5) node[near end,right]{$N$};
\filldraw[fill=gray!30!white,draw=gray,opacity=.4] (-2.8,-2) -- (-2.8,2) -- (2.8,-2) -- (2.8,2) -- cycle;
\filldraw[fill=gray!30!white,draw=gray,opacity=.4] (5.2,-2) -- (5.2,2) -- (10.8,-2) -- (10.8,2) -- cycle;
\node[gray] at (2.55,-1.5) {$C_\alpha$};
\node[gray] at (10.55,-1.5) {$C_\alpha$};
\draw[blue, line width=.75pt] plot [smooth, tension=1] coordinates {(-3,.3) (-2.5,.8) (-2,-.2) (-1,.5) (0,0) (1.5,.7) (2,-.3) (3,.8)} node[blue,above]{$\Gamma_b$};
\node[blue] at (1.3,1) {$\Lambda_b$};
\draw[orange, line width=.75pt] plot [smooth, tension=1] coordinates {(-3,-2) (-2.5,-1.5) (0,0) (1,-.45) (3,.4)};
\node[orange] at (-3.2,-1.5) {$\Gamma_a$};
\node[orange] at (4.8,-1.5) {$\Gamma_a$};
\node[orange] at (1.3,-.85)  {$\Lambda_a$};
\node at (1.2,.3)  {$A$};
\draw[black,thin] (.5,.2) -- (.8,-.45);
\draw[black,thin] (.8,.4) -- (1.1,-.45);
\draw[black,thin] (1.1,.6) -- (1.4,-.35);
\draw[black,thin] (1.4,.7) -- (1.7,-.25);
\draw[blue, line width=.75pt] plot [smooth, tension=1] coordinates {(5,.3) (5.5,.8) (6,-.2) (7,.5) (8,0) (9.5,.7) (10,-.3) (11,.8)}  node[blue,above]{$\Gamma_b$};
\node[blue] at (7,.8) {$\Lambda_b$};
\draw[orange, line width=.75pt] plot [smooth, tension=1] coordinates {(5,-2) (5.5,-1.5) (8,0) (9,-.45) (11,.4)};
\node[orange] at (7,-.9)  {$\Lambda_a$};
\node at (6,-.6)  {$A$};
\draw[black,thin] (5.7,.2) -- (6,-1.2);
\draw[black,thin] (6.1,-.22) -- (6.3,-.95);
\draw[black,thin] (6.4,0) -- (6.65,-.65);
\draw[black,thin] (6.7,.35) -- (7,-.45);
\draw[black,thin] (7.1,.4) -- (7.4,-.3);
\end{tikzpicture}
\caption{The cone $C_\alpha$ and the region $A$. On the left,
$A$ does not touch $\partial\{|\xi-\xi_0|<\delta\}$, while it does on the right.
We can always restrict our attention to the case on the left when $\xi_0$ is a density point of $\cC(f)$.}
\label{f:cono}
\end{figure}

Fix $a,b\in\mathbb S^1$ such that \eqref{nonno} holds and, for $\alpha\in(0,1)$, let
$C_\alpha:=\{v\in\R^2 : |\langle N,v\rangle|< |v|\sin\alpha\}$
be the cone centered at $\xi_0$ with axis parallel to $N^\perp$ and aperture $2\alpha$. 
Since $\Gamma_a,\Gamma_b$ are $\C^1$, there exists $\delta\in(0,1)$ such that
\begin{equation}\label{eq:cont}
\{\xi\in\Gamma_a\cup\Gamma_b : |\xi-\xi_0|<\delta\}\subset C_{\alpha,\delta},
\end{equation}
where we set $C_{\alpha,\delta}=\{\xi\in C_\alpha: |\xi-\xi_0|<\delta\}$.

Let us assume by contradiction that $\cC(f)$ is not a $\C^1$ curve near $\xi_0$.   Then there exists a nonempty connected component $A$ of $C_{\alpha,\delta}\setminus (\Gamma_a\cup \Gamma_b)$ such that, letting   
\[
\Lambda_a=\Gamma_a\cap \partial A,
\quad \Lambda_b:=\Gamma_b\cap \partial A, 
\quad \Lambda_\partial:=\partial \{|\xi-\xi_0|<\delta\}\cap \partial A,
\]
we have 
\begin{equation}\label{eq:hp}
\Lambda_a\neq \emptyset,\quad \Lambda_b\neq \emptyset,\quad \partial A=\Lambda_a\cup\Lambda_b\cup\Lambda_\partial, \quad  \sharp(\Lambda_a\cap\Lambda_b)\leq 2.
\end{equation}
See Figure~\ref{f:cono}. Notice that $A$, $\Lambda_a$, $\Lambda_b$, and $\Lambda_\partial$ depend on $\delta$.
By \eqref{eq:cont} (see also Figure~\ref{fig:conto_trig}), we have 
\begin{equation}\label{eq:area}
  \mathcal L^2(A) \leq \delta^2\tan(\alpha).
  \end{equation}

\begin{figure}
\begin{tikzpicture}[smooth cycle,scale=.8]
\filldraw[fill=green!15!white, draw=green!70!black, line width=.75pt]  (0,0) circle (25mm);
\draw (0,0) node[below] {$\xi_0$};
\node[green!60!black] at (0,-2) (lll) {$\{|\xi-\xi_0|<\delta\}$};
\draw[gray] (-3,0) -- (3,0);
\draw [-stealth](0,0) -- (0,1.5) node[near end,right]{$N$};
\filldraw[fill=gray!30!white,draw=gray,opacity=.4] (-2.8,-2) -- (-2.8,2) -- (2.8,-2) -- (2.8,2) -- cycle;
\node[gray] at (2.55,-1.5) (l) {$C_\alpha$};
\draw [stealth-stealth,red](0,-.2) -- (2.5,-.2) node[midway,below]{$\delta$};
\draw [stealth-stealth,red](2.5,0) -- (2.5,1.8) node[midway,right]{$\delta\tan\alpha$};
\draw[red]  (.8,0) arc(0:35:.8) node[midway,right]{$\alpha$};
\end{tikzpicture}
\caption{Proportions in $C_{\alpha,\delta}$.}
\label{fig:conto_trig}
\end{figure}
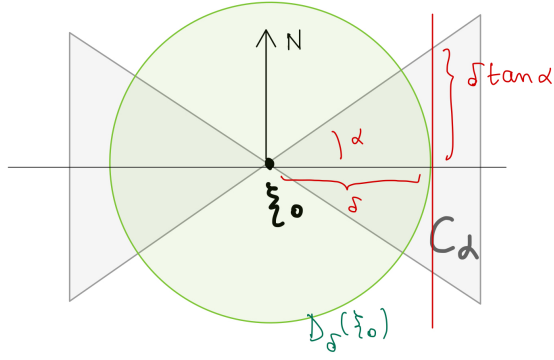

By \eqref{eq:hp} and since $\cC(f)\subset \Lambda_a\cap \Lambda_b$, for $\xi\in \operatorname{int}(\Lambda_a)\cup\operatorname{int}(\Lambda_b)$ we have $F(\xi)\neq 0$, where we endow $\Lambda_a$ and $\Lambda_b$ with their relative topologies. 
 We deduce that
 $F(\xi)=c_a(\xi)a^\perp$ with $c_a(\xi)\ne 0$
 for $\xi\in  \operatorname{int}(\Lambda_a)$
  and 
 $F(\xi)=c_b(\xi)b^\perp$ with $c_b(\xi)\ne 0$ for $\xi\in  \operatorname{int}(\Lambda_b)$.
Using the fact that $\nabla\phi^*$ is positively $0$-homogeneous it then follows that the vector field $\mathcal N: D\setminus \cC(f)\to \R^2$,   $\mathcal N (\xi)  = \nabla \phi^* (F(\xi) )$,
   is constant along $\Lambda_a$ and $\Lambda_b$. Namely,
   \[
   \begin{split}
   \mathcal N (\xi) & = \operatorname{sgn}(c_a) \nabla\phi^*(a^\perp)=:\mathcal N_a ,\quad \xi \in \mathrm{int}(\Lambda_a),
   \\
   \mathcal N (\xi) &=  \operatorname{sgn}(c_b) \nabla\phi^*(b^\perp)=:\mathcal N_b ,\quad \xi \in \mathrm{int}(\Lambda_b).
   \end{split}
   \]
   
 By assumption, and since 
 $\phi^*\in \C^2$, 
 there exists a constant $h \in\R$ such that 
 \begin{equation*}\label{eq:hh}
   \mathrm{div}(\mathcal N (\xi) ) = h ,
   \quad \xi\in D\setminus \cC(f),
 \end{equation*}
 in the strong sense.
 Then by the divergence theorem, and since $A\cap \cC(f) =\emptyset$, we have
 \begin{equation}\label{eq:hLA}
  h  \mathcal L^2(A) =\int_A\mathrm{div} (\mathcal N) dxdy 
  = \int _{\Lambda _a} \langle \mathcal N_a, N_a\rangle d \mathcal H^1 +
   \int _{\Lambda_b} \langle \mathcal N_b, N_b\rangle d \mathcal H^1
   +   \int _{\Lambda_\partial} \langle \mathcal N, N_\partial\rangle d \mathcal H^1,
 \end{equation}
 where $N_a$, $N_b$, and $N_\partial$ are, respectively, the normals to $\Lambda_a$, $\Lambda_b$, and $ \Lambda_\partial$, exterior with respect to $A$.
 For $\alpha\to0^+$ we have 
  \[
  \begin{split}
  & \int _{\Lambda _a}   N_a   \mathcal H^1  = \delta(-N +o(1)),
  \\
  & \int _{\Lambda _b}   N_b   \mathcal H^1  = \delta(N +o(1)),
  \\
&\left|\int _{\Lambda_\partial} \langle \mathcal N, N_\partial\rangle d \mathcal H^1\right|\le C\delta\alpha,
  \end{split}
   \]
   where $o(1)\to 0 $ as $ \alpha \to 0^+$ and $C>0$ denotes a suitable constant.
Now from \eqref{eq:area} we deduce that
\begin{equation}\label{eq:tana}
| \delta\tan(\alpha)  h  | \geq |\langle \mathcal N_b-\mathcal N _a, N\rangle + o(1)|-C\alpha,
    \end{equation}
    that implies  
    $ \langle \mathcal N_b-\mathcal N _a, N\rangle=0$
    in contradiction with \eqref{nonno}.
    
    This proves that $\cC(f)$ is a $\C^1$ curve around any point 
    $\xi_0$ with $\mathrm{rank}(JF(\xi_0))=1$.
\end{proof}

\begin{remark}\label{rmk:cheng}
The statement of Lemma~\ref{l:char1} still holds when in place of assuming $\Sigma$ to have constant $\phi$-curvature we assume that, around a characteristic point $p_0\in\cC(\Sigma)$, its $\phi$-curvature $H_\phi$ satisfies, for a constant $k> 0$,
\[
|\xi-\xi_0| |H_\phi(\xi)| \le k,\quad \xi\in D\setminus\cC(f)
\]
where $f\in\C^2(D)$ is a $z$-graph function so that $p_0=(\xi_0,f(\xi_0))$. 
This is the assumption of \cite[Theorem~3.3]{CHMY}.

In fact in this case, \eqref{eq:hLA} is replaced by
\[\begin{split}
\int _{\Lambda _a} \langle \mathcal N_a, N_a\rangle d \mathcal H^1 +
   \int _{\Lambda_b} \langle \mathcal N_b, N_b\rangle d \mathcal H^1
   +   \int _{\Lambda_\partial} \langle \mathcal N, N_\partial\rangle d \mathcal H^1
   &\le \int_{C_{\alpha,\delta}}\frac{k}{|\xi-\xi_0|} d\xi \leq 4k\alpha\delta
\end{split}
\]
and \eqref{eq:tana} is then replaced by 
\[
4k\alpha  \geq |\langle \mathcal N_b-\mathcal N _a, N\rangle + o(1)|-C\alpha,
\]  
where $o(1)\to 0 $ as $ \alpha \to 0^+$,
yielding the same conclusion.
\end{remark}

\subsection{Characteristic curves in $\phi$-critical surfaces}

Given a surface $\Sigma\subset\H^1$, we call a \emph{characteristic curve on $\Sigma$} any (nontrivial) curve $\Gamma\subset \cC(\Sigma)$. In this section we prove the following result.

\begin{theorem}\label{thm:ccurves}
Let 
 $\Sigma$ be a complete and  oriented surface of class $\C^2$. If $\Sigma$  {is $\phi$-critical with  non-vanishing $\phi$-curvature $h\neq 0$}  then any characteristic curve on $\Sigma$ is either a horizontal line or the horizontal lift of a simple closed curve.
\end{theorem}

For a characteristic curve $\Gamma$ in $\Sigma$ we denote 
its coordinates by  $\Gamma=(\Xi,\zeta)\in\R^2\times\R$. For any $p_0=(\xi_0,z_0)$ on $ \Gamma$,  let $\delta>0$ be small enough to have 
\begin{equation}\label{eq:divi}
\{\xi\in\R^2 : |\xi-\xi_0|<\delta\}\setminus \mathrm 
{supp}(
\Xi)=B^+\cup B^-,
\end{equation} where $B^+,B^-\subset\R^2$ are disjoint open connected sets.
The $\phi$-normal $\cN$   in \eqref{eq:cN} is well-defined in $B^+\cup B^-$.

\begin{lemma}\label{p:cont}
Let $\Sigma$ be a $\C^2$ surface with
constant $\phi$-curvature. With the above notation,  
 the following limits exist
\begin{equation}\label{eq:limits}
\cN^\pm(\xi_0):=\lim_{{B^\pm\ni \xi}\to \xi_0} \cN(\xi)
\end{equation}
and satisfy $\cN^+(\xi_0)=-\cN^-(\xi_0)$.
\end{lemma}

\begin{proof} 
This is a straightforward corollary of  ~\cite[Proposition~3.5]{CHMY}.
\end{proof}

\begin{proposition}\label{p:orto}
Let {$\Sigma$ be a $\phi$-critical surface} of class $\C^2$ and let   
$\Gamma=(\Xi,\zeta)$ be a characteristic curve on $\Sigma$. 
Then for every $p_0=(\xi_0,z_0)$ in $ \Gamma$
 we have
\begin{equation}\label{eq:orto}
{\cN^{\pm}(\xi_0)\in T_{\xi_0}\Xi},
\end{equation}
where $\cN^\pm$ is defined as in Lemma~\ref{p:cont}.
\end{proposition}

\begin{proof}  Let $f\in\C^2(D)$ be a graph function for $\Sigma$  with $\xi_0\in D\subset\R^2$. Without loss of generality we assume $D=\{|\xi-\xi_0|<\delta\}$ and let $D^\pm:=D\cap B^\pm$, where $B^\pm$ are 
{as in}
\eqref{eq:divi}. Let $h\in\R$ be the $\phi$-curvature of $\Sigma$. Since $\Sigma$ is $\phi$-critical, for any  $\varphi \in \C^\infty_c(D)$ we have
\[
\int_{D }\langle\X,\nabla \varphi\rangle\;d\xi = -  \int_D h\varphi\;d\xi
\]
and 
$\operatorname{div}(\X)=h$ pointwise in $D^+\cup D^-$.
Then,
{
denoting by $N_\Xi$ the normal to $\Xi$ pointing towards $D^-$,}
 by the divergence theorem we have  
\[\begin{split}
\int_D h\varphi\;d\xi&=\int_{D^+}\operatorname{div}(\X)\varphi\;d\xi+\int_{D^-}\operatorname{div}(\X)\varphi\;d\xi\\
&=-\int_{D^+\cup D^-}\langle\X,\nabla \varphi\rangle\;d\xi+\int_\Xi\varphi\langle\cN^+,N_\Xi\rangle\;d\mathcal H^1-\int_\Xi\varphi\langle\cN^-,N_\Xi\rangle\;d\mathcal H^1\\
&=\int_D h\varphi\;d\xi+\int_\Xi\varphi\langle\cN^+-\cN^-,N_\Xi\rangle\;d\mathcal H^1.
\end{split}
\]
By Lemma~\ref{p:cont}, this implies that
\[
 \int_\Xi \varphi\langle\cN^ + ,N_\Xi\rangle\;d\mathcal H^1=0
\]
and since   $\varphi$ is arbitrary, this yields the claim.
\end{proof}

 {
\begin{remark}
\label{p:chc2} Under the assumptions of the previous proposition, the characteristic curves $\Gamma=(\Xi,\zeta)$ of $\partial E$ are of class $\C^2$. This can be proved exactly as in  Proposition~4.20 of \cite{RR08} using 
{condition} \eqref{eq:orto}.
  In particular, $\Xi$ is of class $\C^2$.
\end{remark}
}

\subsubsection{Parametrization of constant $\phi$-curvature surfaces around characteristic curves}  

 In this section, we study a $\phi$-{critical}
  surface $\Sigma$ of class $\C^2$ having  constant $\phi$-curvature $h\neq 0$
 near a characteristic curve. 
 Without loss of generality we assume $h>0$.

 We assume $\phi$ to be normalized in such a way that $\phi(1,0)=1$ and we fix a parametrization $\mu:[0,M]\to\R^2$ of $C_\phi$ such that $\phi^\dagger( \dot\mu )= 1$, $\mu([0,M])=C_\phi$, with initial and end-point $\mu(0)=\mu(M)$. We choose the clockwise orientation and we extend $\mu$ to the whole $\R$ by $M$-periodicity.
We have $\mu\in \C^2(\R;\R^2)$ and
\begin{equation}\label{eq:phiphis}
\mu(\tau)=\nabla\phi^*(\dot \mu(\tau)^\perp),\qquad\mbox{ for all }\tau\in \R.
\end{equation}
In fact, letting $\mathcal N(t)=\nabla\phi^*(\dot\mu(t)^\perp)$, we have $\dot\cN=\dot\mu$ as in \eqref{eq:dotNperp}. Equation~\eqref{eq:phiphis} then follows by integration using the fact that $0$ is the center of $C_\phi$.

Let   $\Gamma=(\Xi,\zeta)\in \C^2(I; \Sigma)$  be a characteristic curve 
parameterized in such a way that  
\begin{equation}\label{eq:paraXi}
\phi(\dot \Xi) = 1\quad \text{on }I.
\end{equation}
Locally, $\Gamma$ disconnects $\Sigma$ and there are no other characteristic points of $\Sigma$ close to $\Gamma$, by Lemma~\ref{l:char1}.

According to Corollary~\ref{l:conto},
$\Sigma\setminus\cC(\Sigma)$ admits near $\Gamma$ a Legendre foliation made of horizontal lifts of $\phi$-circles of radius $1/h$, followed in the clockwise sense. Hence, given a point $(\xi_0,z_0)\in \Sigma\setminus\cC(\Sigma)$  near $\Gamma$, there exist $c\in \R^2$ and $\tau\in [0,M]$ such that the horizontal lift of 
\[\xi(s)=c+h^{-1}\mu(\tau+h s)\]
passing through $(\xi_0,z_0)$ at $s=0$
 stays in $\Sigma$ until it meets a characteristic point.
 Here, $c$ is the center of the $\phi$-circle.
Notice that $\nabla \phi^*(\dot\xi(s)^\perp)=\cN(\xi(s))$, so that, by Lemma~\ref{p:cont} and \eqref{eq:orto},  $\nabla \phi^*(\dot\xi(0)^\perp)$ converges to a vector 
collinear to $\dot\Xi(t)$ as $\xi_0$ approaches $\Xi(t)$ for some $t\in I$. 
By \eqref{l:norm1} and \eqref{eq:paraXi}, $\nabla \phi^*(\dot\xi(0)^\perp)$ converges either to $\dot \Xi(t)$ or to $-\dot \Xi(t)$ as $\xi_0$ approaches $\Xi(t)$. 
Since $\Xi$ locally disconnects the plane, we can fix a side
from where $\xi_0$ approaches $\Xi$ and, up to reversing the parameterization of $\Gamma$, we can assume that 
 $\nabla \phi^*(\dot\xi(0)^\perp)$ converges to $\dot \Xi(t)$ as $\xi_0$ converges to $\Xi(t)$. 
Thanks to \eqref{eq:phiphis} and since $\dot\xi(0)=\dot\mu(\tau)$, we deduce that
$\mu(\tau)=\nabla \phi^*(\dot\xi(0)^\perp)$
converges to $\dot \Xi(t)$ as $\xi_0\to \Xi(t)$. 
In particular, the limit direction of $\dot \xi(0)$ as $\xi_0\to \Xi(t)$ is transversal to $\Xi$.

By local compactness of the set of $\phi$-circles with radius $1/h$, the horizontal lift passing through $\Gamma(t)$ at $s=0$ of 
a curve 
$c+h^{-1}\mu(\tau+h s)$ 
with $\mu(\tau)=\dot \Xi(t)$
is a Legendre curve contained in $\Sigma$, for $s$ either in a positive or a negative neighborhood of $0$. To fix the notations, we assume that $s$ is in a positive neighborhood of $0$, the computations being 
equivalent in the other case.  
Moreover, there is no other Legendre curve having $\Gamma(t)$ in its closure and whose projection on the $xy$-plane stays in the chosen side of $\Xi$, since $\tau\in [0,M)$ and $c\in \R^2$ are uniquely determined by 
\[\mu(\tau)=\dot \Xi(t),\qquad 
c=\Xi(t)-h^{-1}\mu(\tau)=\Xi(t)-h^{-1}\dot\Xi(t).\]

It is then possible to  parameterize locally near $\Gamma$ one of the two connected components of $\Sigma\setminus \Gamma$ by Legendre curves using the  function
\begin{equation}\label{eq:gamma}
(t,s)\mapsto\gamma(t,s)=(\xi(t,s),z(t,s))
\end{equation} 
where 
\begin{equation}\label{eq:xi}
\xi(t,s)=h^{-1}\mu(\tau(t)+hs)+\Xi(t)-h^{-1}\dot\Xi(t),\quad t\in I,\ s>0,
\end{equation}
with $\tau$ uniquely defined via the equation  
\begin{equation}\label{eq:tau}
\mu(\tau(t))=\dot\Xi(t),\quad t\in I,
\end{equation}
and $z$ defined by
\begin{equation}\label{eq:z}
z(t,s)=\zeta(t)+\int_0^s \omega(\xi(t,\sigma),\xi_s(t,\sigma))d\sigma.
\end{equation}
As discussed above, we have
\begin{align}
\nabla\phi^*(\xi_s(t, 0)^\perp)&=\dot\Xi(t),\label{eq:NXi}\\
\label{eq:parad}
\phi^\dagger(\xi_s) &=1.
\end{align}

For $t\in I$, we define the \emph{characteristic time} $s(t)$ as 
the first positive time $s>0$ such $\gamma(t,s(t))\in\cC(\Sigma)$.
We will prove later that    such a $s(t)$  exists.
Finally, we let $S:=\{(t,s) : t\in I,\ 0\leq s\leq s(t)\}$ and we consider the surface $\gamma(S)\subset \Sigma$.

\begin{lemma}\label{p:parag}
We have  $\gamma \in  \C^1(S;\Sigma)$ with   $\gamma(\cdot,0)=\Gamma $. 
Moreover, the second order derivatives $\gamma_{ss},
\gamma_{ts},\gamma_{st}$ are well-defined and 
\begin{equation}
\label{eq:schwarz}\gamma_{ts}=\gamma_{st}.
\end{equation}
 \end{lemma}

\begin{proof}
By 
\eqref{eq:xi} and \eqref{eq:z}, we see  that $\gamma_{ss}$ exists and that $\xi_{ts}=\xi_{st}$.
Moreover,
\begin{align*}
z_{st}&=
\omega(\xi_t(t,\cdot),\xi_{s}(t,\cdot))+\omega(\xi(t,\cdot),\xi_{st}(t,\cdot))\\
&=\omega(\xi_t(t,\cdot),\xi_{s}(t,\cdot))+\omega(\xi(t,\cdot),\xi_{ts}(t,\cdot))
=
z_{ts}.
\end{align*}
\end{proof}

On the surface $\gamma(S)$ we consider the vector field 
\begin{equation} 
V(t,s):=\gamma_t(t,s)=(\xi_t(t,s),z_t(t,s))\in\R^3.
\end{equation}
It plays the role of the Jacobi vector field $V$ in \cite[Lemma~6.2]{RR08}. 
The characteristic time $s(t)$ is precisely  the first positive time such that $\langle V(s(t),t), Z\rangle_{\mathcal D}=0$. 
Here, with a slight abuse of notation, $\langle\cdot,\cdot\rangle_{\mathcal D}$ denotes the scalar product that makes $X,Y,Z$ orthonormal. The following computation 
is crucial in what follows. We recall that we are assuming the $\phi$-curvature to be a constant $h\neq 0$.

\begin{lemma}\label{l:contoVZ}
We have the identity 
\[
\langle V(t,s),Z \rangle_{\mathcal D}=
2\big[h^{-2}\omega(\ddot\Xi,\dot\Xi)
+\omega(\dot\Xi-h^{-1}\ddot\Xi,h^{-1}\mu(\tau+hs))\big].
\]
\end{lemma}
\begin{proof}
First notice that 
\begin{equation}\label{eq:VZ}
\langle V,Z\rangle_{\mathcal D}=z_t+\omega(\xi_t,\xi),
\end{equation}
where
\[
z_t(t,s)=z_t(t,0)+\int_0^s\omega(\xi_t(t,\sigma),\xi_s(t,\sigma))\;d\sigma+\int_0^s\omega(\xi(t,\sigma),\xi_{st}(t,\sigma))\;d\sigma.
\]
Using \eqref{eq:xi}, \eqref{eq:tau}, and the skew-symmetry of $\omega$, the above implies
\[
\begin{split}
z_t(\cdot,s)&=\omega(\Xi,\dot\Xi)
+\int_0^s\omega(\dot\Xi-h^{-1}\ddot\Xi+h^{-1}\dot\tau\dot\mu(\tau+h\sigma),\dot\mu(\tau+h\sigma))\;d\sigma\\
&\quad
+\int_0^s\omega(\Xi-h^{-1}\dot\Xi+h^{-1}\mu(\tau+h\sigma),\dot \tau\ddot\mu(\tau+h\sigma))\;d\sigma\\
&=\omega(\Xi,\dot\Xi)
+h^{-1}\omega(\dot\Xi-h^{-1}\ddot\Xi,\mu(\tau+hs)-\mu(\tau))
\\&\quad
+h^{-1}\omega(\Xi-h^{-1}\dot\Xi,\dot \tau\dot\mu(\tau+hs)-\dot \tau\dot\mu(\tau) )\\
&\quad
+h^{-2}\omega(\mu(\tau+hs),\dot \tau\dot\mu(\tau+hs))-h^{-2}\omega(\mu(\tau),\dot \tau\dot\mu(\tau))\\
&=\omega(\Xi,\dot\Xi)
+h^{-1}\omega(\dot\Xi-h^{-1}\ddot\Xi,\mu(\tau+hs))
-h^{-1}\omega(\dot\Xi-h^{-1}\ddot\Xi,\dot\Xi)\\
&\quad+h^{-1}\omega(\Xi-h^{-1}\dot\Xi,\dot\tau\dot\mu(\tau+hs))
-h^{-1}\omega(\Xi-h^{-1}\dot\Xi,\ddot\Xi)
\\&\quad
+h^{-2}\omega(\mu(\tau+hs),\dot\tau\dot\mu(\tau+hs))
-h^{-2}\omega(\dot\Xi,\ddot\Xi)\\
&=\omega(\Xi,\dot\Xi)
-h^{-1}\omega(\Xi,\ddot\Xi)
+h^{-2}\omega(\ddot\Xi,\dot\Xi)
+\omega(\dot\Xi-h^{-1}\ddot\Xi,h^{-1}\mu(\tau+hs))\\
&\quad
+h^{-1}\omega(\Xi-h^{-1}\dot\Xi+h^{-1}\mu(\tau+hs),\dot\tau\dot\mu(\tau+hs)).
\end{split}
\]
Moreover, we have
\[\begin{split}
\omega(\xi_t,\xi)&=
\omega(\dot\Xi-h^{-1}\ddot\Xi+h^{-1}\dot\tau\dot\mu(\tau+hs),\Xi-h^{-1}\dot\Xi+h^{-1}\mu(\tau+hs))
\\
&=h^{-1}\omega(\dot\tau\dot\mu(\tau+hs),\Xi-h^{-1}\dot\Xi
+h^{-1}\mu(\tau+hs))
+\omega(\dot\Xi,\Xi)
\\&\quad
-h^{-1}\omega(\ddot\Xi,\Xi)
+\omega(\dot\Xi-h^{-1}\ddot\Xi,h^{-1}\mu(\tau+hs))
+h^{-2}\omega(\ddot\Xi,\dot\Xi).
\end{split}
\]
Summing up, we obtain the claim.
\end{proof}

We show next that for every $t\in I$, the Legendre curve $s\mapsto \gamma(t,s )$ meets a characteristic point before that $\xi(t,s)$ comes back to the point $\xi(t,0)=\Xi(t)$, i.e., $hs(t)<M$.

\begin{lemma}\label{l:exss}
For any $t\in I$, there exists $s(t)\in(0,M/h)$ such that 
$\langle V(t,s(t)),Z\rangle_{\mathcal D}=0$. 
\end{lemma}

\begin{proof}
For fixed $t$, consider the function $\theta:[0,M]\to\R$, defined by 
\[
\theta(s)=\omega(\dot\Xi-h^{-1}\ddot\Xi,h^{-1}\mu(\tau+hs)).
\] 
By Lemma~\ref{l:contoVZ}, we have that $\langle V(t,s),Z\rangle_{\mathcal D}=0$ if and only if $\theta(s)=b$ with  $b:=h^{-2}\omega(\dot\Xi,\ddot\Xi)$.
The equation $\theta(s) =b$ is certainly satisfied 
for $hs=nM$, $n\in\N$. This follows by the  $M$-periodicity of $\mu$ 
and the fact that $V(t,0)=\dot \Gamma(t)$ is horizontal.

It is enough to consider the case 
 $b\geq0$, the case $b<0$ being analogous. 
By \eqref{eq:tau} we have 
\[
\dot\theta(0)=\omega(\dot\Xi-h^{-1}\ddot\Xi,\dot\mu(\tau))
=\omega(\mu(\tau),\dot\mu(\tau)).
\]
By the  fact that   $C_\phi$ is a convex curve around $0$, it follows that  $\dot\theta(0)\neq 0$.

If $\dot\theta(0)>0$ there exists $s^*\in (0,M/(2h))$ such that $\theta(s^*)>\theta(0)=b$. In this case, by symmetry of $C_\phi$ we have $\mu(\tau+h(s^*+M/(2h)))=-\mu(\tau+hs^*)$, thus implying $\theta(s^*+M/(2h))=-\theta(s^*)<-b\leq 0$. By continuity of $\theta$, we deduce the existence of $\bar s\in(0,M/h)$ satisfying $\theta(\bar s)=b$. We argue in the same way in the case   $\dot\theta(0)<0$.
\end{proof}

We now determine a  quantity that remains constant along the Legendre curves $s\mapsto \gamma(t,s)$.

\begin{proposition}\label{p:constant}
For any $t\in I$ and for all $s\in [0,s(t)]$ we have
\begin{equation}
\langle V(t,s),Z\rangle_{\mathcal D}+h\langle\nabla\phi^\dagger(\xi_s(t,s)),\xi_t(t,s)\rangle = 0.
\end{equation}
\end{proposition}

\begin{proof}
By \eqref{eq:VZ}, 
\eqref{eq:z} and \eqref{eq:schwarz},  we have
\begin{align}
\frac{\partial}{\partial s}\langle V,Z\rangle_{\mathcal D}
&=z_{ts}+\omega(\xi_{ts},\xi)+\omega(\xi_t,\xi_s)
=\frac{\partial}{\partial t}\omega(\xi,\xi_s)+\omega(\xi_{st},\xi)+\omega(\xi_t,\xi_s)\nonumber\\
&=\omega(\xi_t,\xi_s)+\omega(\xi,\xi_{st})+\omega(\xi_{st},\xi)+\omega(\xi_t,\xi_s)=2\omega(\xi_t,\xi_s).\label{eq:claim1cc}
\end{align}

We claim that  
\begin{equation}\label{eq:claim2cc}
h\frac{\partial}{\partial s}(\langle\nabla\phi^\dagger(\xi_s),\xi_t\rangle)=2\omega(\xi_s,\xi_t).
\end{equation}
Indeed, by  Theorem~\ref{l:PMP} and Remark~\ref{PALLO}, we have 
\begin{equation}
\frac{\partial}{\partial s}\nabla\phi^\dagger(\xi_s)=\mathcal H\phi^\dagger(\xi_s)\xi_{ss}=\frac{1}{h}\xi_s^\perp,
\end{equation}
and therefore
\[
\begin{split}
\frac{\partial}{\partial s}\langle\nabla\phi^\dagger(\xi_s(t,s)),\xi_t(t,s)\rangle
&=
\frac{1}{h}\langle\xi_s^\perp,\xi_t\rangle
+\langle\nabla\phi^\dagger(\xi_s),\xi_{st}\rangle
\end{split}
\]
On   differentiating \eqref{eq:parad} w.r.t.\ $t$ we see that $\langle\nabla\phi^\dagger(\xi_s),\xi_{st}\rangle=0$.  This is \eqref{eq:claim2cc}.

Summing up \eqref{eq:claim1cc} and \eqref{eq:claim2cc}, we deduce that the function $\Lambda_{t}(s)=\langle V(t,s),Z\rangle_{\mathcal D}+h\langle\nabla\phi^\dagger(\xi_s(t,s)),\xi_t(t,s)\rangle$
is constant. 
To conclude the proof it is enough to check that $\Lambda_t(0)=0$.
  On the one hand, we have 
$\langle V(t,0),Z\rangle_{\mathcal D}=\langle \dot\Gamma(t),Z\rangle_{\mathcal D}=0$,
since $\Gamma$ is horizontal. On the other hand, since $\nabla\phi^\dagger(v)=-\nabla\phi^*(v^\perp)^\perp$ for any $v\neq 0$, using \eqref{eq:NXi} 
we finally obtain
\[
\langle\nabla\phi^\dagger(\xi_s(t,0)),\xi_t(t,0)\rangle= -\langle\nabla\phi^*(\xi_s(t,0)^\perp)^\perp,\dot\Xi(t)\rangle=0.
\qedhere
\] 
\end{proof}

Since the set $\Gamma_1:=\{\gamma(t,s(t)) : t\in I\}$ is made of characteristic points, it is either an isolated point or a nontrivial characteristic curve (Lemma~\ref{l:char1}). We will see in  the proof of Theorem~\ref{thmi:class}, contained in Section \ref{s:proof2}, 
that if $\Gamma_1$ were an isolated characteristic point, then the same would be true for $\Gamma$. We stress that the argument leading to such a conclusion does not rely on  the characterization of $\Gamma$ provided in this section.  
We then have  that $\Gamma_1:=\{\gamma(t,s(t)) : t\in I\}$ is a nontrivial characteristic curve.

\begin{proposition}\label{p:sconstant}
The function  $t\mapsto s(t)$ is constant.
\end{proposition}

\begin{proof}
Let $t\in I$. Since $\langle V(t,s(t)),Z\rangle_{\mathcal D}=0$, the point $\gamma(t,s(t))$ is characteristic for $\Sigma$. Then, by Lemma~\ref{l:char1} and Remark~\ref{p:chc2}, $\Gamma_1$ is a $\C^2$ characteristic curve.
 By the implicit function theorem,
 the function $t\mapsto s(t)$ is $\C^1$-smooth and for $t\in I$ we have
\[
\dot\Gamma_1(t)=V(t,s(t))+\dot s(t)\gamma_s(t,s(t)).
\]
The curve $\Xi_1$ obtained by projecting $\Gamma_1$ on the $xy$-plane then  satisfies 
\[
\dot\Xi_{1}(t)=\xi_t(t,s(t))+\dot s(t)\xi_s(t,s(t)).
\]
Since $\gamma(t,s(t))\in\cC(\Sigma)$, by Proposition~\ref{p:orto},
and using the fact that $\nabla\phi^\dagger(v)=-\nabla\phi^*(v^\perp)^\perp$ for any $v\neq 0$ we have 
\[
\langle\nabla\phi^\dagger(\xi_s(t, s(t))),\dot\Xi_1(t)\rangle=-\langle\nabla\phi^*(\xi_s(t, s(t))^\perp)^\perp,\dot\Xi_1(t)\rangle=0.
\]
Therefore we obtain 
\begin{equation}\label{eq:ppp}
0=\langle\nabla\phi^\dagger(\xi_s(t, s(t))),\xi_t(t,s(t))\rangle+\dot s(t)\langle\nabla\phi^\dagger(\xi_s(t, s(t))),\xi_s(t,s(t))\rangle,
\end{equation}
where, by Proposition~\ref{p:constant},
\[
\langle\nabla\phi^\dagger(\xi_s(t, s(t))),\xi_t(t,s(t))\rangle=0,
\]
and moreover, by \eqref{eq:parad},
\[
\langle\nabla\phi^\dagger(\xi_s(t, s(t))),\xi_s(t,s(t))=\phi^\dagger(\xi_s(t,s(t)))=1.
\]
Equation \eqref{eq:ppp} thus implies $\dot s= 0$, which concludes the proof.
\end{proof}

We are now ready to prove Theorem~\ref{thm:ccurves}.

\begin{proof}[Proof of Theorem~\ref{thm:ccurves}]{Without loss of generality we assume $h>0$.}
By Remark~\ref{p:chc2}, $\Gamma$ is of class $\C^2$ and we denote by $I$ an interval of parametrization of $\Gamma=(\Xi,\zeta)$ satisfying \eqref{eq:paraXi}. We consider the parametrization $\gamma$ given by Lemma~\ref{p:parag}. By Proposition~\ref{p:sconstant}  the characteristic time $s(t)$ is constant on $I$ and we let $s(t)=\bar s\in\R$. 
Since $\langle V(t,\bar s),Z\rangle_{\mathcal D}=0$, by Lemma~\ref{l:contoVZ} we thus have 
\[
h^{-2}\omega(\ddot\Xi(t),\dot\Xi(t))
+\omega(\dot\Xi(t)-h^{-1}\ddot\Xi(t),h^{-1}\mu(\tau(t)+h\bar s))=0.
\]
Using \eqref{eq:tau}, the last equation reads
\begin{equation}\label{eq:gatto}
\dot\tau\omega(\dot\mu(\tau),\mu(\tau)-\mu(\tau+h\bar s))=h\omega(\mu(\tau+h\bar s),\mu(\tau)).
\end{equation} 
If the right-hand side is $0$ at some $t\in I$, then    $\mu(\tau(t))$ and $\mu(\tau(t)+h\bar s)$ are parallel by definition of $\omega$ (cf.~\eqref{eq:omega}). Since $h\bar s\in (0,M)$ by Lemma~\ref{l:exss}, the only possible choice is $h\bar s=M/2$. Plugging such choice into the left-hand side and using the fact that $\mu(\tau+M/2)=-\mu(\tau)$, we obtain
\[
2\dot\tau\omega(\dot\mu(\tau),\mu(\tau))= 0\quad\text{on}\quad I.
\]
This implies that $\dot\tau=0$ on $I$ and therefore that $\tau$ is constant on $I$. By \eqref{eq:tau} we deduce that  
$\dot \Xi$ is constant on $I$ implying that $\Xi$ is a straight line.

We are now left to consider the case $h\bar s\in(0,M)$, $h\bar s\neq M/2$, 
so that  $\omega(\mu(\tau(t)+h\bar s),\mu(\tau(t)))\neq 0$ for every $t\in I$. Equation~\eqref{eq:gatto} reads
\[
\dot\tau= f(\tau)\quad\text{with}\quad f(\tau):=\frac{h\omega(\mu(\tau+h\bar s),\mu(\tau))}{\omega(\dot\mu(\tau),\mu(\tau)-\mu(\tau+h\bar s))}.
\]
For the sake of simplicity, assume $0\in I$. 
Notice that 
$f$ 
is $M/2$-periodic 
and 
of class $\C^1$ as a function of $\tau$, therefore $f$ is bounded. Hence, given $\tau_0\in\R$ satisfying $\mu(\tau_0)=\dot \Xi(0)$, there is a unique maximal solution  $\tau$ defined on the whole $\R$ to the differential equation with  the initial condition $\tau(0)=\tau_0$. Since $h\bar s\in (0,M)$, $h\bar s\neq M/2$, we have $f(\tau)\neq 0$, yielding 
that 
$|\dot \tau|$ is lower bounded by a positive constant.
To fix the ideas, assume that  $\operatorname{sign}(\dot\tau)=1$. Then, there exists $T_0>0$ such that $\tau(T_0)=\tau_0+M/2$. 
We claim that 
\begin{equation}
\label{eq:pert}
\tau(t+T_0)=\tau(t)+\frac{M}{2}\qquad\text{for all\ } t\in\R.
\end{equation}
This follows from the fact that $\tau_1(t):=\tau(T_0+t)$ and $\tau_2(t):=\tau(t)+M/2$ for $t\in\R$ solve the same Cauchy problem 
$
\dot\tau(t)=f(\tau)$,
$\tau(0)=\tau_0+M/2$.
Then, by \eqref{eq:pert}, $M$-periodicity of $\mu$, and \eqref{eq:tau}, we have for every $t\in\R$ 
\[
\dot\Xi(t+2T_0)=\mu(\tau(t+2T_0))=\mu(\tau(t)+M)=\mu(\tau(t))=\dot\Xi(t),
\]
i.e., $\dot\Xi$ is $2T_0$-periodic.
This implies that $\Xi$ is also $2T_0$-periodic. Indeed, for $t\in\R$ we have
\[
\begin{split}
\Xi(t+2T_0)-\Xi(t)&=\int_t^{t+2T_0}\dot\Xi(\sigma)\;d\sigma=\int_t^{t+2T_0}\mu(\tau(\sigma))\;d\sigma\\
&= \int_t^{t+T_0}\mu(\tau(\sigma))\;d\sigma+ \int_{t}^{t+T_0}\mu(\tau(\sigma+T_0))\;d\sigma\\
&= \int_t^{t+T_0}\mu(\tau(\sigma))\;d\sigma- \int_{t}^{t+T_0}\mu(\tau(\sigma))\;d\sigma=0,
\end{split}\]
where we have used again the symmetry of $C_\phi$ and \eqref{eq:pert}.

We are   left to show that $\Xi(\bar \sigma)\neq \Xi(\bar t)$ for any $0\le \bar \sigma<\bar t<2T_0$. Assume  that $\Xi(\bar \sigma)=\Xi(\bar t)$
for some $0\le \bar \sigma<\bar t\le 2T_0$.  Then we have   $0=\int_{\bar \sigma}^{\bar t}\dot\Xi(t)\;dt= \int_{\bar \sigma}^{\bar{t}}\mu(\tau(t))\;dt$. Now, letting $v:=\mu(\tau(\bar \sigma))$, by the symmetry of $C_\phi$  the function 
\[
\sigma\mapsto \int_{\bar \sigma}^\sigma\langle\mu(\tau(t),v)\rangle\;dt
\]
is monotone increasing for $\sigma\in[\bar \sigma,\bar \sigma+T_0]$ and decreasing for $\sigma\in[\bar \sigma+T_0,\bar \sigma+2T_0]$. Hence, the equation $\int_{\bar \sigma}^{\bar{t}}\mu(\tau(t))\;dt=0$ implies $\bar \sigma=0$ and $\bar t= 2T_0$.
\end{proof}

\subsection{Characteristic set of isoperimetric sets}
In this section we apply the previous results to the study of the characteristic set of $\phi$-isoperimetric sets.
As a Corollary of Theorem~\ref{thm:ccurves} we have the following 
\begin{corollary}\label{l:char1old} Let $\phi^*$ be of class $\C^2$ 
and let $E\subset \H^1$ be a $\phi$-isoperimetric set of class $\C^2$.
Then $\cC(E)$ consists of isolated points.
Moreover, for every $p_0=(\xi_0,z_0)\in \cC(E)$ and every $f$ such that $p_0\in \mathrm{gr}(f)\subset \partial E$, 
we have $\operatorname{rank}(JF(\xi_0))=2$.
\end{corollary}

\begin{proof} By Remark~\ref{rem:bdd}, we know that $\partial E$ is bounded. Therefore we exclude the possibility that $\cC(\partial E)$ contains complete (unbounded) lifts of simple curves. 
\end{proof}

\begin{lemma}\label{l:ch2} Let $\phi^*$ be of class $\C^2$ and $E\subset \H^1$ be a $\phi$-isoperimetric set of class $\C^2$.
Let $p_0\in \cC(E)$. There exists $r>0$ such that for  $p\in \partial E \cap B(p_0,r) $, $p\neq p_0$, the maximal horizontal lift of the $\phi$-circle in $\partial E$
through $p$ meets $p_0$.
\end{lemma}

\begin{proof}
The surface $\partial E\cap B(p_0,r)$ is the $z$-graph
of $f\in \C^2(D)$ and $p_0 = (\xi_0, f(\xi_0)) $ with $\cC(f) \cap  \{|\xi-\xi_0|<r\} = \{\xi_0\}$.
Let $\Theta_\xi\subset D
$ be the maximal $\phi$-circle (integral curve of $F^\perp$) passing though $\xi\in D\setminus\{\xi_0\}$. Notice that  the radius of $\Theta_\xi$
does not depend on $\xi$, as it follows from
 Corollary~\ref{l:conto}.
If $\xi_0\notin \Theta_\xi$, then
the normal  vector $\mathcal N_\xi =\nabla \phi^*(F)$ is continuously defined on $\Theta_\xi$.

Assume that there exists a sequence of such $\xi$ with $\xi\to \xi_0$. By an elementary compactness argument it follows that
there exists a $\phi$-circle $\Theta$ passing through $\xi_0$ and there exists a normal $\mathcal N$ that is continuously defined along $\Theta$ and, in particular, through $\xi_0$. Outside $\xi_0$ we have $\mathcal N = \nabla \phi^*(F)$.

Let $b\in\R^2$ the unit   vector tangent to $\Theta$ at $\xi_0$.
Then we have
\[
F(\xi_0+ tb) = F(\xi_0) + t J F(\xi_0) b+ o(t) = t J F(\xi_0) b+ o(t),
\]
 with $JF(\xi_0) b \neq 0$, because $JF(\xi_0)$ has rank $2$ by Lemma~\ref{l:char1}. Since $\nabla \phi (-v) = - \nabla \phi(v)$, for $v\in\R^2\setminus\{0\}$, it follows that
 \[
\begin{split}
 &  \lim _{t\to 0^+} \nabla \phi^*(F(\xi_0+ tb)) = \nabla \phi^*(J F(\xi_0) b),
 \\
 & \lim _{t\to 0^-} \nabla \phi^*(F(\xi_0+ tb)) = -\nabla \phi^*(J F(\xi_0) b).
 \end{split}
 \]
 This contradicts the continuity of $\mathcal N$ along $\Theta$ at $\xi_0$.
\end{proof}

 \section{Classification of $\phi$-isoperimetric sets of class $\C^2$}\label{s:proof}

\subsection{Construction of $\phi$-bubbles}

Let $\phi$ be a norm in $\R^2$ that we normalize by  $\phi(1,0)=1$.
For $\xi_0\in \R^2$ and $r>0$, $\phi$-circles are defined in \eqref{eq:qCircle} and we let the \emph{$\phi$-disk} of radius $r$ and center $\xi_0$ be 
\[
        D_\phi(\xi_0,r)  = \{\xi  \in \R^2: \phi(\xi-\xi_0) <r  \}.
\]
We also let $
C_\phi(r)= C_\phi(0,r)$,
$C_\phi= C_\phi(1)$ and
$D_\phi(r)= D_\phi(0,r)$,
$D_\phi= D_\phi(1)$.

The circle $C_\phi$ is   a Lipschitz curve and
we denote by  $L=L_\phi > 0$ its
Euclidean  length.
We parametrize $C_\phi$ by arc-length through   $\kappa \in \mathrm{Lip} \big( [0,L];\R^2 \big)$ such that $\kappa([0,L]) = C_\phi$ with initial and end-point
$\kappa(0) = \kappa(L) = (-1,0)$.  We choose the  anti-clockwise orientation and we extend $\kappa$ to the whole $\R$ by $L$-periodicity.
Then we have $ \kappa \in \mathrm{Lip}  (\R;\R^2)$.

 The map
$\xi  : \R^2 \to \R^2$,
        $\xi(t,\tau) = \kappa(t) + \kappa(\tau)$,
is in   $\mathrm{Lip} (\R^2;\R^2)$.
We restrict $\xi$ to the domain
\[
D = \big\{ (t,\tau) \in \R^2: \tau \in [0,L], t \in [\tau +{L}/{2}, \tau +{3L}/{2}] \big\}.
\]
Notice that $\xi (\tau+{L}/{2}, \tau) =  \xi (\tau + {3L}/{2}, \tau) = 0$ for any $\tau\in [0,L]$.
We  define the function $z\in \mathrm{Lip}(D)$,
\begin{equation}\label{eq:Psi3}
                 z(t,\tau) =   \int _{\tau + {L}/{2}}^t
                        \omega \big (\xi(s,\tau), \xi_s(s,\tau)\big) ds
                                 .
\end{equation}
The map $\Phi:D\to\R^3$ defined by $\Phi = (\xi,z)$ is Lipschitz continuous. Moreover, $\Phi$ is $\C^k$ if $\phi$ is $\C^k$.

We define the Lipschitz surface
$        \Sigma_\phi= \Phi  (D) \subset \R^3
$ and  call $S=\Phi (\tau + {L}/{2},\tau) = 0\in \Sigma_\phi$ the south pole of $\Sigma_\phi$ and
$N=\Phi(\tau + {3L}/{2},\tau) = (0,0,z(\tau + {3L}/{2},\tau))$ the north pole.
We call  the bounded region  $E_\phi \subset \R^3$  enclosed by $\Sigma_\phi$ the $\phi$-bubble.
 $E_\phi$ is a topological ball and it is  the candidate solution to the $\phi$-isoperimetric problem. When $\phi$ is the Euclidean
 norm in the plane,  the set $E_\phi$ is the well-known Pansu's ball.

\subsection{Classification of $\phi$-isoperimetric sets of class $\C^2$}\label{s:proof2}

We are ready to prove the main theorem of the paper. 

\begin{proof}[Proof of Theorem~\ref{thmi:class}] The set $E$ is bounded and connected, by Remark \ref{rem:bdd}. We may also assume that it is open.
It follows from Corollary~\ref{l:conto} (and from the analogous result for $x$-graphs and $y$-graphs based on 
Remark~\ref{rem:x-graphs2}) that, out of  the characteristic set $\cC(E)$, the surface 
$\partial E$ is foliated by horizontal lifts of $\phi$-circles. 
Then $\cC(E)$ contains  at least one point, since otherwise, $\partial E$ would contain an unbounded curve, contradicting the boundedness of $E$.

Let $f\in \C^2 (D)$, with $D\subset \R^2$ open,
be a maximal function such that $\mathrm{gr}(f)\subset \partial E$ and $\mathcal C(f)\neq \emptyset$.
We may assume that $0\in \mathcal C(f)$,  $f(0)=0$ and that $E$ lies above the graph of $f$ near $0$.
Around the characteristic point $0$, the function $f$ must have the structure described in Lemma \ref{l:ch2}. It follows that, up to a dilation, we have $\mathrm{gr}(f)\subset \partial E_\phi$.

The maximal domain for $f$  must be $D= D_\phi(2)$.
Otherwise, at each point $\xi\in \partial D\setminus \partial D_\phi(2)$ the space 
$T_{(\xi,f(\xi))}\partial E=T_{(\xi,f(\xi))}\partial E_\phi$ is not vertical, contradicting the maximality of $D$. 
This shows that the graph of $f$ is the `lower hemisphere' of $\partial E_\phi$.

Up to extending $f$ by continuity to $\partial D$, we have $(\xi,f(\xi))\notin\cC(E)$
for each $\xi\in\partial D$. Hence there exists a $\phi$-circle passing through $0$ whose horizontal lift stays in $\partial E$ and passes through $(\xi,f(\xi))$. The collection of all the maximal extensions of such horizontal lifts completes the upper hemisphere of $\partial E_\phi$, thus implying that $\partial E_\phi\subset \partial E$.
Moreover, since $\partial E$ is $\C^2$, we deduce that $\partial E_\phi$ is a connected component of $\partial E$.

In conclusion we have proved that $\partial E$ is the finite union of  boundaries of $\phi$-bubbles having the same curvature. By connectedness of $E$ this concludes the proof.
\end{proof}

In the next proposition, we show that $\phi$-bubbles have the same regularity as $\phi$ outside the poles.

\begin{proposition}\label{l:sm}
If  $\phi$ is strictly convex
and of class $\C^k$, for some  $k\geq 1$, then the set $\Sigma _\phi \setminus\{ S,N\}$
is an embedded surface of class $\C^k$.
\end{proposition}

\begin{proof} If the Jacobian of $\Phi$ has rank 2 at the point $(t,\tau)\in D$,
then $\Sigma_\phi$ is an embedded surface of class $\C^k$ around the point $\Phi(t,\tau)$.
A sufficient condition for this is $
\det                 J\xi (t,\tau) \neq 0$.
The Jacobian of $\xi: D\to\R^2$ satisfies
        \begin{equation*}\label{eq:JPsi}
\det                 J\xi (t,\tau) = 0\quad \textrm{if and only if}\quad
        \dot{\kappa}(t) = \pm \dot{\kappa}(\tau).
        \end{equation*}
The case $\dot{\kappa}(t) = -\dot{\kappa}(\tau)$ is equivalent to
${\kappa}(t) = -{\kappa}(\tau)$, by the strict convexity of the norm.
This is in turn equivalent to $t=\tau+L/2$ or $t = \tau+3L/2$.
In the former case we have $\Phi(t,\tau) = S$, in the latter  $\Phi(t,\tau) = N$.

We are left to consider the case
$\dot{\kappa}(t) = \dot{\kappa}(\tau)$.
By strict convexity of $\phi$, this implies  ${\kappa}(t) = {\kappa}(\tau)$, that is equivalent to $t =\tau + L$.
In this case, we have
$\xi(t,\tau) = 2 \kappa(\tau) \in  C_\phi(2)$
The point $\Phi(t,\tau)$ is on the `equator' of $\Sigma_\phi$.

        We study the regularity of $\Sigma_\phi$ at points
        $\Phi(\tau+L,\tau)$.
        The height   $z (\tau+L,\tau)$ does not depend on $\tau$ because it  is half the area of the disk   $D_\phi$. 
        It follows that $
                0 =  \partial_\tau \big (  z  (\tau+L,\tau) \big)  =
                z_t (\tau+L,\tau) +
                z_\tau (\tau+L,\tau)
        $
        and this implies that
        \begin{equation}\label{eq:dPsi3t}
        z_t  (\tau+L,\tau) \neq
        z_\tau (\tau+L,\tau),
        \end{equation}
as soon as we prove that the left-hand side does not vanish.
Indeed, differentiating  \eqref{eq:Psi3} we
        obtain
\begin{equation*}         z_t (\tau+L ,\tau) =
        2 \omega\big(  \kappa(\tau),\dot\kappa(\tau)   \big) \neq 0,
\end{equation*}
because
 $\kappa(\tau)$ and $\dot{\kappa}(\tau)$ are not  proportional.

From
$\dot{\kappa}(\tau+L) = \dot{\kappa}(\tau) \neq 0$
and \eqref{eq:dPsi3t}, we deduce that the Jacobian matrix $ J\Phi (\tau+L,\tau)$ has rank 2. This shows that $\Sigma_\phi$ is of class $\C^k$
        also around the `equator'.
\end{proof}

In general, $\phi$-bubbles are not of class $\C^2$ and not even of class $\C^1$, e.g., in the case of a crystalline norm. Even when $\phi$ is regular,
there may be a loss
of regularity at the poles of $E_\phi$. 

In \cite[Theorem~5.7]{PozueloRitore}, Pozuelo and Ritor\'e show that $\partial E_\phi$ is an embedded surface of class $\C^2$ under the assumption that $\phi$ is of class $\C^2_+$. Here, we say that a norm $\phi$ in $\R^2$ is of class $\C^r_+$ for $r\in\mathbb N\cup \{\infty\}$, $r\geq 2$, if $\phi \in \C^r(\R^2\setminus\{0\})$ and $\phi$-circles have strictly positive curvature.
In the next proposition we show that in fact the assumption in \cite[Theorem~5.7]{PozueloRitore} is equivalent to the fact that $\phi$  and $\phi^*$ are $\C^2$.
\begin{proposition}\label{prop:c2+} 
Let $\phi$ be a norm in $\R^2$ with dual $\phi^*$, and let $r\in\N\cup\{+\infty\}$, $r\ge 2$. Then the following are equivalent
\begin{itemize}
\item[(i)] $\phi$ and $\phi^*$ are of class $\C^r$;
\item[(ii)] $\phi$ is of class $\C^r_+$.
\end{itemize}
\end{proposition}

\begin{proof}
We start by observing that, following \cite[Section~2.5]{Schneider93}, $\phi$ is of class $\C^r_+$ if and only if $\phi$ is of class $\C^r$ and the map $\nu_\phi:C_\phi \to\mathbb S^1$, $\nu_\phi(\xi)={\nabla\phi(\xi)}/{|\nabla\phi(\xi)|}$,
is a $\C^{r-1}$ diffeomorphism. 
The fact that (ii) implies (i) then follows from \cite[Cor.~5.3]{PozueloRitore}. 

We show that (i) implies (ii). 
Consider the $\C^r$-diffeomorphism $S_{\phi^*}:C_{\phi^*}\to \mathbb S^1$, $S_{\phi^*}(\xi)=\xi/|\xi|$, with inverse $S_{\phi^*}^{-1}(u)=u/\phi^*(u)$, $u\in\mathbb S^1$, and let $\cN_{\phi}:C_\phi\to C_{\phi^*}$, $\cN_\phi(\xi)=\nabla \phi(\xi)$. Here we are using the fact  that for any $\xi\in\R^2\setminus \{0\}$ we have $\nabla\phi^*(\xi)\in C_\phi$, as follows \emph{e.g.,} by~\cite[Theorem~1.7.4]{Schneider93}.
Then we have that 
\begin{equation}\label{eq:nu}
\nu_\phi=S_{\phi^*}\circ \cN_\phi,
\end{equation}
and the statement follows once we prove that $\cN_\phi$ is a diffeomorphism of class $\C^{r-1}$.
To this purpose, observe that the maps $\cN_{\phi^*}$ and $\cN_\phi$ are both of class $\C^{r-1}$ by (ii). We are then left to prove that $\cN_{\phi^*}$ and $\cN_\phi$ are inverse one of the other. In fact, since $\phi^*$ is $\C^r$, $r\ge 2$, we have that $\phi$ is strictly convex (arguing e.g.\ as in~\cite[Theorem~26.3]{Rockafellar}). Hence, by~\cite[Corollary~1.7.3]{Schneider93}, $\nu_\phi$ is invertible and satisfies 
\[
\nu_\phi^{-1}\left(\frac{\eta}{|\eta|}\right)=\cN_{\phi^*}(\eta),\quad \eta\in\R^2\setminus\{0\}.
\]
Using this relation together with \eqref{eq:nu}, we thus obtain for any $\eta\in C_{\phi^*}$:
\[
\cN_\phi(\cN_{\phi^*}(\eta))=\cN_\phi(\cN_\phi^{-1}\circ S_{\phi^*}^{-1}\left({\eta}/{|\eta|}\right))=\eta.
\]
By duality we also have that $\cN_{\phi^*}(\cN_\phi(\xi))=\xi$ for any $\xi\in C_\phi$, thus concluding the proof. 
\end{proof}

\section{The isoperimetric problem for general norms}\label{s:cr}

In the case of crystalline norms, the first order necessary conditions satisfied by an isoperimetric set
are not sufficient to reconstruct its structure, even assuming sufficient regularity.
In this section we show that   the $\phi$-isoperimetric problem for a general norm -- in particular for a crystalline norm -- can be approximated by the isoperimetric problem for  smooth norms.

By \cite[Theorem~5.7]{PozueloRitore} (and Proposition~\ref{prop:c2+}), we know that if $\phi$ and $\phi^*$ are of class $\C^\infty$, then the $\phi$-bubble $E_\phi$ is of class $\C^2$.
In this section, we show that
the validity of Conjecture~\ref{iBenbow}
implies the
$\phi$-isoperimetric property for 
 the $\phi$-bubble of any (crystalline) norm. 

\subsection{Smooth approximation of norms in the plane}

We start with  the mollification of a norm.

\begin{proposition}\label{p:approx} Let $\phi$ be a norm in $\R^2$.
Then, for any $\varepsilon>0$ there exists a  norm $\phi_\varepsilon$ of class $\C^\infty_+$, 
such that for all $\xi\in\R^2$ we have
\begin{equation} \label{piff}
                (1-\eta( \epsilon) ) \phi_\epsilon(\xi)\leq \phi(\xi)\leq
                (1+\eta( \epsilon)) \phi_\epsilon(\xi),
\end{equation}
and $\eta( \epsilon) \to0$ as $\epsilon\to0^+$.
\end{proposition}

\begin{proof}
For $\varepsilon>0$, we introduce the smooth mollifiers $\varrho_\varepsilon:\R\to \R$,  supported in $[-\varepsilon\pi,\varepsilon\pi]$ defined by
\[
\varrho_\varepsilon(t)=
\begin{cases}
c_\varepsilon \exp\Big(  {\frac{\pi^2\epsilon^2}{t^2-\pi^2\epsilon^2}}\Big) &\text{if }|t|<\pi,\\
0&\text{if }|t|\geq \pi,
\end{cases}
\]
where $c_\varepsilon$ is chosen in such a way that $\int_\R\varrho_\varepsilon(t)\;dt=1$.
Following \cite{DWM99,MBC06}, we define the function $\psi_\varepsilon:\R^2\to[0,\infty)$ letting
\[
 \psi_\varepsilon(\xi):=\int_{\R}\varrho_\varepsilon(t)\phi(R_t \xi)\;dt,
\]
where $R_t$ denotes the anti-clockwise rotation matrix of angle $t$.  The function $\psi_\varepsilon$ is a $\C^\infty$ norm.
On the circle $\mathbb S^1=\{\xi\in\R^2:|\xi|=1\}$, the norms $\psi_\epsilon$ converge uniformly to $\phi$ as $\epsilon\to0^+$. So our claim \eqref{piff} 
{with $\eta(\epsilon)\to 0$} holds with  $\psi_\epsilon$ replacing $\phi_\epsilon$, by the positive $1$-homogeneity of norms.

We let $\phi_\varepsilon:\R^2\to [0,\infty)$ be defined by
\[
\phi_\varepsilon(\xi):=\sqrt{\psi_\varepsilon(\xi)^2+\varepsilon|\xi|^2},\quad \xi\in\R^2.
\]
This is a $\C^\infty$ norm in $\R^2$ and 
\eqref{piff} is satisfied with $\eta(\epsilon)\to 0$.
The {unit} $\phi_\epsilon$-circle {centered at the origin} is the $0$-level set of the function
\[
F_\epsilon (\xi)=\psi_\varepsilon^2(\xi)+\varepsilon|\xi|^2-1, \quad \xi\in\R^2.
\]
Since the Hessian matrix of the squared Euclidean norm is proportional to the identity matrix $I_2$ and $\psi_\varepsilon^2$ is convex, we have that $\mathcal H\! F_\epsilon   \geq 2\varepsilon I_2$ in the sense of matrices. Then the curvature $\lambda_\epsilon$ of a {unit} $\phi_\epsilon$-circle satisfies 
\begin{equation*}
\label{eq:kurv}\begin{split}
\lambda_\epsilon  &=\frac{\langle \mathcal H F_\epsilon \nabla F_\epsilon ^\perp,\nabla F_\epsilon  ^\perp\rangle}{|\nabla F_\epsilon |^3}\geq \frac{ 2\varepsilon}{|\nabla F_\epsilon |}>0.
\qedhere
\end{split}
\end{equation*}
\end{proof}

\subsection{Crystalline $\phi$-bubbles as limits  of smooth isoperimetric sets} Let $\phi$ be any norm in $\R^2$ and  let $\{\phi_\varepsilon\}_{\varepsilon>0}$ be the smooth approximating norms found in Proposition~\ref{p:approx}.

Given a Lebesgue measurable set $F\subset \R^2$,
from \eqref{piff}
and from the definition of perimeter (Definition~\ref{d:per}), we have
\begin{equation}
\label{Jim}
(1-\eta(\epsilon) )
\P_{\phi}(F)  \leq
\P_{\phi_\epsilon}(F) \leq (1+\eta(\epsilon) )
\P_{\phi}(F)  .
\end{equation}

 The $\phi_\epsilon$-circles $C_{\phi_\epsilon}$ converge in Hausdorff distance to the circle $C_\phi$.
This implies that the $\phi_\epsilon$-bubbles $E_{\phi_\epsilon}$ converge in the Hausdorff distance to the limit bubble   $E_\phi$.
This in turn implies the convergence in $L^1(\mathbb H^1)$, namely,
\begin{equation}\label{tesoro}
\lim_{\epsilon \to0^+} \mathcal L^3( E_{\phi_\epsilon}\Delta E_\phi)=0,
\end{equation}
where $\Delta$ denotes the symmetric difference of sets.

\begin{proof}[Proof of Theorem~\ref{thmi:appr}]
Let $F\subset \mathbb H^1 $ be any Lebesgue measurable set with $0<\mathcal L^3(F) <\infty$.
Assuming the validity of Conjecture \ref{iBenbow},
$E_{\phi_\epsilon}$ is isoperimetric for any $\epsilon>0$.
So using twice   \eqref{Jim} 
we find 
\[
\operatorname{Isop_{\phi}}(F)
\geq  \frac{
\operatorname{Isop_{\phi_\epsilon}}(F)}{  1+\eta(\epsilon)}
\geq  \frac{\operatorname{Isop_{\phi_\epsilon}}(E_{\phi_\epsilon})
}{  1+\eta(\epsilon)}
\geq  \frac{1-\eta(\epsilon) }{  1+\eta(\epsilon)}
\operatorname{Isop_{\phi}}(E_{\phi_\epsilon}).
\]
By the lower semicontinuity of the perimeter with respect to the $L^1$ convergence and from \eqref{tesoro}, we deduce that
\[
\liminf_{\epsilon\to0^+}  \operatorname{Isop_{\phi }}(E_{\phi_\epsilon})\geq
\operatorname{Isop_{\phi}}(E_{\phi }),
\]
and using the fact that $\eta(\epsilon)\to 0$ we conclude that
$
\operatorname{Isop_{\phi}}(F)
\geq  \operatorname{Isop_{\phi}}(E_\phi).
$
\end{proof}

\end{document}